\documentclass[11pt]{amsart}

\usepackage[shortlabels]{enumitem}
\usepackage[normalem]{ulem}
\usepackage{arcs}
\usepackage{cancel}
\usepackage{graphicx}
\usepackage{amsmath}
\usepackage{amsmath, amssymb}
\usepackage{mathrsfs}
\usepackage{amsfonts}
\usepackage{amscd}
\usepackage{arcs}

\usepackage{pdfpages}
 \usepackage{pifont}

\newcommand{\nc}{\newcommand}
\nc{\G}{{\Gamma}}
\nc{\BC}{{\mathbb C}}
\nc{\BQ}{{\mathbb Q}}
\nc{\BR}{{\mathbb R}}
\nc{\BZ}{{\mathbb Z}}
\nc{\BP}{{\mathbb P}}
\nc{\BN}{{\mathbb N}}
\nc{\BM}{{\mathbb M}}
\nc{\fH}{{\mathfrak{H}}}

\newtheorem{thm*}{\bf Theorem}[section]

\newtheorem{ques}{\bf Question}[section]
\newtheorem{thm}{\bf Theorem}[section]
\newtheorem{cor}{\bf Corollary}[section]
\newtheorem{lemma}{\bf Lemma}[section]
\newtheorem{prop}{\bf Proposition}[section]
\newtheorem{defn}{\bf Definition}[section]

\newtheorem{rem}{\bf Remark}[section]

\renewcommand{\>}{\rangle}

\newcommand{\w}{\omega}

\newcommand{\dist}{\operatorname{dist}}

\renewcommand{\phi}{\varphi}

 \newcommand{\Ir}{{\text{Irr}}}
\newcommand{\Po}{{\text{Pol}}}
\newcommand{\su}{{\text{supp}}}
\renewcommand{\Im}{\operatorname{Im}}
\renewcommand{\Re}{\operatorname{Re}}

\newcommand{\RR}{{\mathbb R}}
\newcommand{\CC}{{\mathbb C}}
\newcommand{\QQ}{{\mathbb Q}}
\newcommand{\NN}{{\mathbb N}}

\newcommand{\ZZ}{{\mathbb Z}}

\newcommand{\C}{{\mathcal C}}
\newcommand{\OO}{{\mathcal O}}
\newcommand{\p}{{\mathcal P}}
\newcommand{\K}{{\mathcal K}}
\newcommand{\R}{{\mathcal R}}
\newcommand{\D}{{\mathcal D}}

\newcommand{\U}{{\mathcal U}}
\newcommand{\M}{{\mathcal M}}
\newcommand{\BB}{{\mathcal B}}

\begin{document}

\title{On a Fekete-Szeg\"o theorem }
\author{ Th\'er\`ese Falliero }

\address{Th\'er\`ese Falliero \\ Avignon Universit\'e, Laboratoire de math\'ematiques d’Avignon (EA 2151),  F-84018 Avignon, France }

   \email{therese.falliero@univ-avignon.fr}
   

\maketitle 

\hspace{7cm}

\begin{abstract}
 We consider again a classical theorem relating capacities and algebraic integers and the question of  the simultaneous approximation of $ n-1$ different complex numbers by conjugate algebraic integers of degree $n$.
\end{abstract}

\bigskip

{\it This is a preliminary version of a further more complete paper, which will be co-written by Ahmed Sebbar.}

\section{Introduction}

A classical theorem of Fekete and Szeg\"o \cite{fs} says that, given a compact set $K$  in the complex plane   having logarithmic capacity $\C(K)$, then
\begin{enumerate}[(a)] 
\item If $ \C(K)< 1$, there is an open set $U$ containing  $K $    such that there are only finitely
many algebraic integers $\alpha$ having all of their conjugates in $U$, and
\item if $\C(K)\geq 1$, then every open $U$ containing $K $   contains infinitely many such $ \alpha$.
\end{enumerate}
Furthermore Fekete and Szeg\"o  proved also that if $K$ is a compact set in the complex plane, stable under complex conjugation and having logarithmic capacity $\C(K)\geq 1$, then every neighborhood of $K$ contains infinitely many conjugate sets of algebraic integers.
In \cite{Ennola} V. Ennola  solved a question raised by R. M. Robinson that if  $\Delta$ is any real interval of length greater than 4, then
for any sufficiently large $n$ there exists an irreducible monic polynomial of degree $n$ with integer coefficients all of whose zeros lie in $\Delta$ .

We should emphasize that many diophantine inequalities are reduced to the existence of lattice points in some convex body \cite[Chapter III]{Cassel}. In this direction, it is remarkable that T. Chinburg \cite{Chinburg} reduces the proof of Fekete theorem to an application of Minkowski's Convex Body theorem \cite[Chapter III]{Cassel} that we recall for later use.

This problem is actually related to a precise form of the Stone-Weierstrass theorem. This classical theorem states that every continuous function defined on a closed interval $[a, b]$ can be uniformly approximated by polynomials. A more general statement is if $X$ is a compact Hausdorff topological space and if $ C(X)$ is the algebra of real-valued continuous functions $f:X\rightarrow \BR$, then a subalgebra  ${\mathcal A}\subset C(X)$  is dense if and only if it separates points. 

The question is for which compact set continuous functions can be  approximated by polynomials with integer coefficients? This question is a major one in approximation theory and the literature is very extensive \cite{ferg}.
 Let $f$ be a continuous real-valued function defined on $[0,1]$, then the sequence of polynomials $(p_n)$ defined by
 \[p_n(x)= \sum_{\nu=0}^n f\left(\frac{\nu}{n} \right)\left(\begin{matrix}n\\\nu \end{matrix} \right) x^{\nu}(1-x)^{n-\nu}
 \]
converges uniformly to $f$. This is therefore a constructive proof of the  Stone-Weierstrass theorem. It is due to Bernstein. We deduce from this result \cite[Theorem 5]{ferg2} that for a continuous real-valued function $f$ on the unit interval  $I= [0, 1]$ to be uniformly approximable by polynomials with integer coefficients it is necessary and sufficient that it be integer-valued at both $0$ and $1$.

 As was pointed out in \cite[Lemma 1]{ferg2}, If $q$ is a non constant polynomial with integer coefficients and $I$ is an interval of length at least four, then $\Vert q\Vert \leq2$, where $\Vert .\Vert $ is the supremum norm. Hence, clearly, the approximation by polynomials with integer coefficients on a set $E$ is related to to the capacity of $E$. Consequently \cite[Theorem 2 ]{ferg2} if the capacity $\C(I)\geq 1$ the only functions that are uniformly approximable in $I$ by polynomials with integer coefficients are these polynomials themselves. On the other hand one can prove  the surprising result that any  $f\in L^2([a,\,b]),\; b-a<4$ can be approximated on the interval $[a, b]$ by polynomials with integer coefficients.

\begin{thm}[P\'al]
If $f$ is continuous on $[-a,a], |a|<1$, and $f(0)$ is an integer then $f$ may be uniformly approximated thereon by polynomials with integer coefficients.
\end{thm}
The proof can be found in \cite{pal}, see also \cite{ferg}, p.1.

These problems are at the core of the approximation theory by polynomials with integer coefficients. 

As Chebyshev polynomials will be of great importance in this paper, we recall some facts about them.
For a given compact $K\subset {\BC}$, a 
monic polynomial $p(z)\in {\BC}[z]$ of degree $n \geq 0$ is called the 
$n$-th polynomial of least deviation (from zero) or the Chebyshev polynomial of
degree n
 if $||p||_K \leq||q||_K$ 
for any monic polynomial $q(z)\in {\BC}[z]$ of degree $n$, where 
$||p||_K= \max_{z\in K}\{|p(z)|\}$.  That such a polynomial exists and is unique is a classical result (see for example, \cite{Ts} Theorem III.23).
For the following, let us observe, in the case of
 the interval $[-2\ ,\ 2]$,  that the theory of the Chebyshev polynomials $T_n(X)$
with $T_n (\cos\theta)= 2\cos(n\theta),~ n\in\BN$, is a part
of the spectral analysis of the matrix $ A=(a_{i,j})_{1\leq i,j},~
a_{i,i+1}=a_{i,i-1}=1$ and $a_{i,j} =0$ otherwise.

The classical Chebyshev result states that for an interval $[a,b]$,
$$\inf_{Q}\|Q\|_{\infty} = 2\left(\frac{b-a}{4}\right)^n,$$
where the infimum is taken over all monic polynomials $Q$ of degree $n$ with real coefficients.

 Hilbert showed  in \cite{Hilbert} that 
$$\inf_{Q}\|Q\|_{L^2([a,b])} \leq C\sqrt{n}\left(\frac{b-a}{4}\right)^{n/2},$$
where the infimum is over monic polynomials of degree $n$ with integer coefficients, and $C > 0$ is an absolute constant.
Fekete showed in \cite{Fekete} the more flexible
\[\inf_{Q} \Vert Q(z) \Vert_{\infty}\leq 2^{1-\frac{1}{n+1}} (n+1) \left( \frac{b-a}{4}\right)^{n/2},
\]
where the infimum is again over monic polynomials of degree $n$ with integer coefficients.

For a set $V\subset\CC$  let $\Po_V$ be the set of  monic polynomials in $\ZZ[X]$ of degree at least 1 such that all their roots lie in $V$.
Let $E\subset \RR$ be a finite union of segments with $\C(E)>1$ and $\mu_E$ be its equilibrium measure. We have the following theorem (\cite{Se} Theorem 1.6.2).
\begin{thm}[Serre]
There exists a sequence of polynomials $P_n\in {\rm Pol}_E$ such that $\mu_{P_n}\to \mu_E$.

\end{thm}

The main objective of this work is the following natural question from   the Fekete-Szeg\"o theorem, suggested by J.P. Serre in \cite{Se}.
\begin{ques}\label{s}
Let $K$ be a compact of $\CC$ stable under complex conjugation, of capacity $\C(K)\geq 1$ and $U$ an open set containing $K$. Is there a sequence of polynomials $P_n\in {\rm Pol}_U$ such that $\mu_{P_n}\to \mu_K$?
\end{ques}
In this paper we obtain the following result.
\begin{thm}
 Let $K$ be a compact of $\BC$, symmetric with respect to the real axis,
with $\C(K) \geq 1$. 
If $U $ is an open set containing $K$, there is a sequence $(P_n)$ of monic polynomials with integer coefficients whose roots are in $U$ and are such that the associated zeros counting measure $\mu_{P_n}$ converge weakly to the equilibrium measure $ \mu_K $ of $ K$.
\end{thm}

{\bf Organization of the paper:} Very succinctly, we give the definitions in the first section, then we introduce the counting measures in the second section. Minkowski's theorem will be discussed in section three. The fourth section recall some facts in the case of a compact of $\RR$. The fifth section
 is devoted to certain related results. The sixth and seventh sections are devoted to the introduction of certain Riemann surfaces, in relation to certain Jacobi matrices and the solution of the Serre question.

\section{Definitions}\label{defs}

\subsection{Approximation on intervals}
If $\mu$ is a finite Borel measure on $\CC$ with compact support, its logarithm potential is the function  $\Phi_\mu:\CC\rightarrow (-\infty,+\infty ]$  defined by
$$\Phi_\mu(z)=\int \ln(|z-w|^{-1})\, d\mu(w)\, .$$
This integral converges if $z\not\in \su(d\mu)$, and since $d\mu$ has compact support, $\ln(|z-w|^{-1})$ is uniformly bounded below for $(z,w)\in \su(d\mu) \times \su(d\mu)$, so the integral for each $z\in \su(\mu)$ either converges or diverges to $+\infty$, in which case we set $\Phi_\mu(z)=+\infty$.

Potentials enter naturally in studying growth of polynomials as $n\to \infty$. For if
\[P_n(x)=\prod_{j=1}^n(x-x_j^{(n)})\]
then 
\[\frac{1}{n}\ln|P_n(x)|=-\Phi_{\mu_{P_n}}(x)\]
where
$$\mu_{P_n}=\frac{1}{n}\sum_{j=1}^n\delta_{x_j^{(n)}}$$
is the counting measure for the zeros $x_j^{(n)}$, also called the normalized density of zeros. The function $\Phi_\mu(z)$ is bounded below on $\su(\mu)$, so
\[I(\mu)=\int \Phi_\mu(z)\, d\mu(z)=\int\int_{K\times K} \ln(|z-w|^{-1}) d\mu(z)d\mu(w)\]
is either finite or diverges to $+\infty$. $I(\mu)$ is called the potential energy of $\mu$ or, for short, the energy of $\mu$.

Consider a compact $K\subset \CC$. We consider all probability measures ${\mathcal M}_{+,1}(K)$ on $K$. We say $K$ has capacity zero if and only if $I(\mu)=\infty$ for all $\mu\in {\mathcal M}_{+,1}(K)$. We set \[\displaystyle v(K)=\inf_{\mu}I(\mu)\] where $\mu$ runs over all positive probability measures supported in $K$. Then the capacity of $K$ is defined as $\C(K)=e^{-v(K)}$. Its logarithm $v(K)$ is called the logarithm capacity of $K$.

For a compact $K$ with non-zero capacity there exists a unique positive probability measure $\mu$, such that $I(\mu)=\ln \C(K)^{-1}$. This measure $\mu=\mu_K$ is called the equilibrium measure of $K$.

Before continuing further, we wish to recall some definitions that  
will be need.

Let us also recall that if $(\mu_n)_n$ and  $\mu_\infty$ are probability measures on a compact Hausdorff space $X$,  $(\mu_n)_n$ converges weakly  to  $ \mu_\infty$ if
$$\int f d\mu_n \rightarrow \int fd\mu \quad \quad \text{as }  n\to \infty$$
for every $f\in \C^0(X)$, function continuous on $X$.

Denote by $\M(K)=\C^0(K)^*$ the set of all measure on $K$. The $\sigma(\M(K),\C^0(K))$-topology is the weakest topology on $\M(K)$ in which the maps $x\mapsto \<y,x\>$ of $\M(K)$ to $\CC$ are continuous for all $y\in \C^0(K)$. By the Banach-Alaoglu theorem, the unit ball in $\M(K)$ is compact in the $\sigma(\M(K),\C^0(K))$-topology. $\M_{+,1}(K)$ is closed in the unit ball, so it is compact too.

 Moreover if $f$ a real valued function defined on a topological space $E$,
$f$ is lower semi continuous at $a$ if 
$$f(a)=\liminf_{x\to a} f(x)\, .$$ We have the following property of the potential energy
$$\mu \longrightarrow I(\mu)$$ is weakly lower semi continuous. The lower semi continuity means 
\[\mu_n\to \mu \Rightarrow \liminf I(\mu_n)\geq I(\mu)\]
 equivalently 
 \begin{enumerate}
 \item $I^{-1}((-\infty,a])$ is closed for all $a$\\
 \item   $I^{-1}((a, \infty])$ is open for all $a$. 
 \end{enumerate}

 Finally, given a bounded set ${\mathcal E}$ in the complex plane, 
we denote by
${\mathcal E}(r)$ the $r-$neighborhood of ${\mathcal E}$.
\begin{defn}\label{H}
 If
 ${\mathcal E}_1$ and  ${\mathcal E}_2$ are two bounded sets in ${\BC}$, 
the difference between ${\mathcal E}_1$ and  ${\mathcal E}_2$ is the 
smallest $r$
such that ${\mathcal E}_1(r)$ contains ${\mathcal E}_2$ and 
${\mathcal E}_2(r)$ contains ${\mathcal E}_1$. It is the Hausdorff distance between ${\mathcal E}_1$ and ${ \mathcal E}_2$.
\end{defn}
We will denote this difference by 
$\delta({\mathcal E}_1, {\mathcal E}_2)$, it is small if and only if 
${\mathcal E}_1$ and  ${\mathcal E}_2$ are (almost) super imposable. We will say that 
${\mathcal E}_1$ is near  ${\mathcal E}_2$ (and reciprocally).
If ${\mathcal E}$ 
is a compact set
and $({\mathcal E}_{\nu})$ a family of sets such that 
$\delta({\mathcal E}, {\mathcal E}_{\nu})$ tends to zero as ${\nu}$ tends to
$\infty$ we will say simply that $({\mathcal E}_{\nu})$ tends to 
${\mathcal E}$.

We then deduce the lemma
\begin{lemma} \label{delta}

If $\forall\varepsilon > 0$, there exists $n_0$ such that for all $n \ge n_0$,
\(
K_n \subset K(\varepsilon).
\), and $\C(K_n)\to \C(K)$ then $\mu_{K_n}\to \mu_K$. This is in particular the case when
  $\delta(K_n,K)\to 0$ and $\C(K_n)\to \C(K)$.
\end{lemma}
\begin{proof}

If $\forall\varepsilon > 0$, there exists $n_0$ such that for all $n \ge n_0$,
\(
K_n \subset K(\varepsilon),
\)
in particular, $\mu_{K_n}$ and $\mu_K$ have their supports in a compact set,
then there exists $\eta$, a weak limit point of $\mu_{K_n}$, $\mu_{K_{n_j}}\to \mu_K$.
By the preceding property, $\operatorname{supp} \eta\subset K$.
In effect, let $K' \subset \CC\backslash K$, a compact set. Let $C\subset K'$ a compact such that there exists $f$ be a continuous function such that
\[
0 \le f \le 1, \qquad f \equiv 1 \text{ on } C, \qquad \operatorname{supp} f \subset K' .
\]
Let $\varepsilon > 0$ such that
\[
K' \subset \CC \backslash K(\varepsilon).
\]
Then for $n \ge n_0$, we have $K' \subset \CC\backslash K_n$ and
\[
\forall j, n_j\geq n_0, \int f \, d\mu_{K_{n_j}} = 0 \;\;\Rightarrow\;\; \int f \, d\mu = 0.
\]
As
\[
\mu(C) \le \int f \, d\mu = 0,
\]
we conclude that $\mu(C)=0$, 
$C \subset \CC \setminus \operatorname{supp} \mu$ and $K' \subset \CC \setminus \operatorname{supp} \mu$.
Hence $\operatorname{supp} \mu \subset K$.
We repeat the arguments of \cite{S-T} Appendix and \cite{Simon}. See also \cite{Simon} Thm 4.5.7.

%
%
%
%
%
 By lower semi continuity of the energy $I$,
\begin{eqnarray*}
I(\eta)&\leq &\liminf I(\mu_{K_{n_j}})\\ 
& =& \lim \ln(\C(K_{n_j})^{-1})\\
&=&\ln(\C(K)^{-1}),
\end{eqnarray*}
so $\eta=\mu_K$, that is, $\mu_{K_{n_j}}\to \mu_K$.
\end{proof}

Note that the converse does not hold: one can have $\mu_{P_n} \to \mu_K$ weakly without $C(\{\text{zeros of }P_n\}) \to C(K)$. Indeed, since each zero set $\{\text{zeros of }P_n\}$ is finite, it has logarithmic capacity zero, even though the associated normalized counting measures can converge weakly to $\mu_K$.

\section{Algebraic integers with all conjugates in a given compact.}
For a set $V\subset\CC$  let $\Po_V$ be the set of  monic polynomials in $\ZZ[X]$ of degree at least 1 such that all their roots lie in $V$. If $z$ is an algebraic integer, a root of a polynomial 
$\displaystyle P(X)\in \Po_V $, then all the conjugate of $z$ are in $V$ and we  say that $z$ is totally in $V$.

 Let $\Ir_V$ be the set of irreducible monic polynomials in $\ZZ[X]$ of degree at least 1 such that all their roots lie in $V$. for such a polynomial $P(X)$ of degree $g$ let $\mu_P$ be the corresponding probability measure supported in its roots, $\displaystyle \mu_P=\frac{1}{g}\sum_{i=1}^g \delta_{x_i}$. Now let $K\subset \CC$ be compact. There are two quite different cases  \cite{Se},\,\cite{Rumely}, depending on the capacity of $K$.
\begin{enumerate}
\item If  $\C(K)<1$, then $\Ir_K$ is finite.
\item If $K\subset \CC$ is ${\rm Gal}(\bar{\QQ}/\QQ)$-stable and $\C(K)\geq 1$, then for any open $U, K \subset U$, the set $\Ir_U$ is infinite. 
\end{enumerate}
In particular if $E\subset \RR$ is a union of finite number of segments and $\C(E)>1$, then $\Ir_E$ is infinite.


\subsection{Precisions on the properties of $K$}
Let $U\subset \CC$ be a set and let $U^*$ be   the so called ``symmetric kernel"  of $U$ consisting of those points of $U$ which belong, together with their conjugates to $U$.  So $U^*$ is symmetric with respect to the real axis and naturally $\C(U^*)\leq \C(U)$.

M.Fekete \cite{Fekete},  \cite{Chinburg} proved that if $K$ is a compact of $\CC$ such that $\C(K^*)<1$, then there is only a finite number of  irreducible algebraic equations with integer coefficients of the form
$$z^n+a_1 z^{n-1}+...+a_{n-1}z+a_n=0$$ whose roots lie all in $K^*$.

\begin{thm}[Minkowski’s theorem]
Suppose $K$ to be a symmetric, convex, bounded subset $\BR^d$. If ${\rm vol}(K) > 2^d$, then $K$ contains at least one lattice point other than $0$.
\end{thm}
There is an extension to general lattices $\displaystyle \Lambda= \BZ u_1\oplus \cdots \oplus \BZ u_d$, where $\{ u_1, \cdots, u_d\}$ is a basis of $\BR^d$. We define ${\rm vol}(\Lambda) $ as the volume of the
parallelotope
\[\left\{  \sum_{i=1}^d  \alpha_i u_i,\quad 0\leq \alpha_i\leq 1 \right\}
\]

\begin{thm}[Minkowski’s theorem for general lattices]
Suppose $\Lambda$ to be a lattice and $K$ to be a bounded symmetric convex subset in $\BR^d$. If  ${\rm vol}(K) > 2^d\, {\rm det}\Lambda$, then $K$ contains at least a point of $\Lambda$ different from $0$.
\end{thm}

We must perhaps insist that Minkowski's theorem as well as Motzkin's theorem on simultaneous approximation (which in turn depends on two theorems of Kronecker. The first one \cite[p.159]{Mot}) is at the heart of  the diophantine  approximation  and then at the heart of the approximation by polynomials with integer coefficients. The second one \cite{Kronecker}
state that if an algebraic integer $\alpha$ and all of its conjugates are  in the closed unit disk ${\mathbb D} := \{z \in \BC : |z|\leq  1\}$, then it is either $\alpha= 0$ or it is root of unity.
 This is apparent at \cite[Chap III]{Cassel}, Feruguson \cite[Theorem 1.1]{ferg} and Chinburg \cite{Chinburg}. For the sake of completeness and in order to see how the different idea articulate
we give an idea of the proof of the first part of Fekete-Szeg\"o theorem. 

We detail  an application of a Chinburg's theorem. See also \cite{a}, p.24.
 For $\displaystyle  a\in \BR^{n+1}$ we define the polynomial $ f_a(z)= a_0+\cdots a_nz^n $. Let $K$ be a compact
such that $\C(K)<1 $, then (see \cite{Chinburg}) we consider
$$f_n(K) = \left\{a = (a_0, a_1, \ldots, a_n) \in \mathbb{R}^{n+1}\setminus\{0\} \;:\; \max_{z\in K}|f_a(z)| < 1\right\}.$$
Let ${\psi}_n$ the euclidien measure on $\BR^{n+1}$, theorem 1.2 in \cite{Chinburg} says
\begin{equation}\label{chin}
\lim n^{-2}2\ln\psi_n(f_n(K))=-\ln\C(K)\, .
\end{equation}


We note that, setting $\widetilde{f}_a(z)=\frac{f_a(z)}{a_n}$, we have $K\subset \widetilde{f}_a^{-1}(D(0,\frac{1}{|a_n|}))$ and we deduce that $\C(K)\leq \frac{1}{|a_n|^{1/n}}$, that is
$$|a_n|\leq \frac{1}{\C(K)^n}.$$
Then in the case $\C(K)>1$, $\lim a_n=0$ and $\forall \epsilon>0,\exists N, \forall n\geq N$, $f_n(K)\subset \BR^n\times[-\epsilon,\epsilon]$ and $\lim\psi_n(f_n(K))=0$, we can compare with \ref{chin}.

In the case $\C(K)<1$, \ref{chin} allows to show Fekete's theorem. For this we apply 
Minkowski's theorem: let $C$ open symmetric convex set of $\BR^{n+1}$, if $\psi_n(C)>2^{n+1}$, then $C$ contains a point with integer coordinates, different from $0$.

Here $f_a(K)\subset \BR^{n+1}$, as $f_{-a}(z)=f_a(z)$, $f_n(K)$ is the symmetric.

As  $f_{ta+(1-t)b})=tf_a+(1-t)f_b$, for $t\in[0,1]$, $f_n(K)$ is a convex set.

Let's verify that $f_n(K)$ is an open bounded set.

Let $a\in f_n(K)$, as $K$ is a compact set and $f_a$ continue, $m_a=\max_{z\in K}|f_a(z)|<1$. Let $\rho$ be the radius of the smallest disc contained in $K$ and centered at $0$. Let $M=\sum_{k=0}^n \rho^k$ and $\epsilon\in\BR^{n+1}$ such that $||\epsilon||=(\sum_0^n|\epsilon_i|^2)^{1/2}$, $||\epsilon||<\frac{1-m_a}{M}$. Then, with $f_{a+b}=f_a+f_b$, we deduce that $D(a,\epsilon)\subset f_n(K)$ and $f_n(K)$ is open.

We can add that $f_n(K)$ is bounded obtaining different estimates on $|a_n|$. Several methods allow to justify that, for $0\leq k\leq n$, $|a_k|\leq (n+1)!\rho^{\phi(n)}\C(K)^{-\frac{n(n+1)}{2}}$, where $\phi(n)$ is a function depending on $n$ only.

Consider $n+1$ points of $K$, $z_0,z_1,...,z_n$ such that
$|\Pi_{0\leq i<j\leq n}(z_j-z_i)|=\sup_{x_k\in K}|\Pi_{0\leq i<j\leq n}(x_j-x_i)|:=V_{n+1}$; such a set is called, Fekete's set. With the preceding notations, we denote by $w_i=f_a(z_i)$. Then, we have a linear system in the $a_i$:
$$\left\{\begin{array}{ccc}a_0+a_1z_0+...+a_nz_0^n&=&w_0\\...&=&...\\a_0+a_1z_n+...+a_nz_n^n&=&w_n\end{array}\right.$$

Let's denote the determinant of the system , $\left|\begin{array}{cccc}1&z_0&...&z_0^n\\
.& .&...&.\\1&z_n&...&z_n^n\end{array}\right|$ by $\det V=V(z_0,...,z_n)=\Pi_{0\leq i<j\leq n}(z_j-z_i)$, so that $|V(z_0,...,z_n)|=V_{n+1}$.

We have, for example,
$$a_0=\left|\begin{array}{cccc}w_0&z_0&...&z_0^n\\
.& .&...&.\\w_n&z_n&...&z_n^n\end{array}\right|/ V(z_0,...,z_n)  \, .$$ 
More generaly,  $(V^{-1})_{ij}=\displaystyle\frac{\widetilde{V}_{ij}(V)}{V(z_0,...,z_n)}$, where $\widetilde{V}_{ij}(V)$ is the "classical mineur" $(-1)^{i+j}|V_{i}^j|$.

If $V_{i}^j=(\alpha_{kl})_{1\leq k,l\leq n}$ then $|V_{i}^j|=\sum_{\sigma\in S_n}\epsilon(\sigma)\alpha_{1\sigma(1)}...\alpha_{n\sigma(n)}$ and $|\alpha_{l\sigma(l)}|\leq \rho^{\sigma(l)}$.

Then $\left||V_i^j|\right|\leq n!\rho^{\sum_{k\not=j}k}$ and as $a_i=\frac{1}{\det V}\sum_{l=0}^n\widetilde{V}_{ij}(V)w_l$, we have $|a_i|\leq \displaystyle\frac{n!}{\det V}\sum_{j=0}^n\rho^{\frac{n(n+1)}{2}-j}$.

Denoting $d_{n+1}=V_{n+1}^\frac{2}{n(n+1)}$, $(d_n)$ tends  decreasing to $\C(K)$, then
$\det V=V_{n+1}\geq \C(K)^{\frac{2}{n(n+1)}}$ and 
$\forall i, 0\leq i\leq  n$, $|a_i|\leq n!\rho^{\frac{n(n+1)}{2}}\displaystyle\frac{\rho^{n+1}-1}{ \rho^{n+1}-\rho  }\C(K)^{-\frac{n(n+1)}{2}}
$.

\bigskip

Another inequality can be obtained noting that
\begin{eqnarray*}
& &V(z_0,...,z_{i-1},X,z_{i+1},...,z_n)=\sum_{l=0}^n\widetilde{V}_{il}(V)X^l\\
&=&\Pi_{0\leq l<k<i}(z_k-z_l)\Pi_{l=0}^{i-1}(X-z_l)\Pi_{k=i+1}^n(z_k-X)\Pi_{i<l<k\leq n}(z_k-z_l)
\Pi_{l=0}^{i-1}(X-z_l)\Pi_{k=i+1}^n(X-z_k)\\
&=&X^n-s_1
X^{n-1}+...+(-1)^{(n-j)}s_{n-j} X^j+...+(-1)^ns_n\, .
\end{eqnarray*}
We have
$s_{n-j}=\sum_{1\leq i_1<i_2<...<i_{n-j}\leq k} z_{i_1}z_{i_2}...z_{i_{n-j}}$. As $\forall l, |z_l|\leq \rho \quad |s_{n-j}|\leq(\binom{n}{j}) \rho^{n-j}$ and $|\widetilde{V}_{ij}(V)|\leq V_n (\binom{n}{j}) \rho^{n-j}$.

Finally $|\frac{\widetilde{V}_{ij}(V)}{\det(V)}|\leq \frac{V_n}{V_{n+1}} (\binom{n}{j}) \rho^{n-j}$

In conclusion $f_n(K)$
 is an open symmetric convex set, with large volume for a large $n$. By Minkowski's theorem, $f_n(K)$ contains some $a= (a_0,a_1,\cdots,a_n)\in \BZ^{n+1}\setminus{0}\}$. We fix a such $a$ and we consider the open set
\[U= \{ z\in \BC, \mid f_a(z)\mid<1 \}.
\]
If $\alpha$ and its conjugates are contained in $U$, then $f_a(\alpha)$ and its conjugates are contained in  the unit disk $\mathbb D$. By the second Kronecker theorem  $\alpha$ is one of the many roots of the polynomial $f_a(z)$.
\begin{rem}
The idea of using Minkowski's convex body theorem in this context goes back to Hilbert \cite{Hilbert}. We can rephrase, in a classical way \cite{Cassel}, what we said on the proof of the first part of Fekete-Szeg\"o theorem: Any convex body of volume  at least $2^n$ contains at leasr $2^n+1$ integral points. In particular the system of linear inequalities
\[
\bigl\vert \sum_{k=1}^n a_{k,m}x^k \bigr\vert \leq b_m,\quad 1\leq m\leq n\]
with
\[{\rm det} \left(a_{k,m}\right)_{1\leq k,m\leq n}\neq 0,\quad \prod_{m=1}^n b_m \geq \bigl\vert {\rm det} \left(a_{k,m}\right)_{1\leq k,m\leq n}\  \bigr\vert
\]
has a nonzero integral solution.
\end{rem}

Now let $K$ be a compact of $\CC$ of capacity $\C(K)\geq 1$. Let $U$ a neighborhood of $K$ then $U^*$ is a neighborhood of $K^*$ (with the convention that the empty set is the neighborhood of the empty set). If $\C(K^*)<1$ then by continuity of the capacity, there exists $V$ a neighborhood of $K$ such that $\C(V^*)<1$. From what precedes, there exists only a finite number of algebraic integers with all its conjugate in $V$. Then to show that there exists an infinity of algebraic integers totally in $U$, we can assume $\C(K^*)\geq 1$ and finally the hypothesis of $K$ symmetric with respect to the real axis is natural.

\section{On a compact symmetric with respect to the real axis}
\subsection{The boundary}
First of all, $K$ being a metric compact set, many assumptions can be made on $\partial K$.
\begin{defn}
We say that $K$ has a continuous boundary when $K$ is a non empty union of connected components non reduced to a point.
\end{defn}
Recovering $K$ with a finite number of small enough  closed balls, we can assume that the boundary $\partial K$ is continuous. 

From \cite[Proposition p. 18 ]{BG} we can suppose $K$ with regular boundary of class $C^\infty$. In fact let $U$ be an open set containing  $K$, there exists a $C^\infty$ function $\phi$ in $\RR^2$ such that  
\begin{enumerate}
\item $\phi=1$ on $K$,
\item $\text{supp }(\phi)\subset U$.
\end{enumerate}
Then $K\subset \text{supp}(\phi)\subset U$ with $\text{supp}(\phi)$ a compact  set with regular $C^\infty$ boundary. 

Let us  recall that a  analytic Jordan curve is a closed curve $\Gamma$ in $\CC$ which possesses a neighborhood $V$ and a conformal map $\xi$ from $V$ on $\{\alpha<|z|<\beta\}$, such that the image of $\Gamma$ by $\xi$ is the circle $\{|z|=r\}$, $\alpha<r<\beta$. We can also suppose that $\partial K$ is a set of analytic Jordan curves. To see this, a first method is to use \cite[Theorem G]{fs} . Let $K(\rho)$ be the $\rho$-neighborhood of $K$ ($\rho>0$), there exists $\rho$ sufficiently small so that $K(\rho)\subset U$. From the preceding theorem, there exists a domain defined by a lemniscate containing $K$ and contained in $K(\rho)$: for $n\geq n_1(\rho)$, $\{z, |w_n(z)|\leq \nu_n\}$ where 
\[w_n(z)=\prod_{k=1}^n (z-\zeta_k^{(n)}) (z-\overline{\zeta_k^{(n)}}),\quad \nu_n=\max_K|w_n(z)|.\] In conclusion this compact set is invariant under complex conjugation and its boundary consists in analytic Jordan curves.

A second method can be found in \cite[p.144]{as}. As an open set of $\CC$, $U$ is a natural Riemann open surface, there exists a sequence of regular subregions $(U_n)$, such that $\overline{U_n}\subset U_{n+1}$ and $\displaystyle U=\bigcup_{i=1}^\infty U_n$.

We recall that $\Omega$ is regularly imbedded if $\Omega$ and its exterior have a common boundary which is a 1-dimensional submanifold. A regularly imbedded subregion of a Riemann surface is thus bounded by analytic curves, then $\forall n$ $\partial U_n$ consists in analytic Jordan curves.

$K$ being a compact in $U$, it can be recovered by a finite number of $U_n$. Then there exists $N$ such that $K\subset U_N$. Then $K\subset \overline{U_N}\subset U_{N+1}\subset U$ and $\overline{U_N}$ is a compact with analytic boundary.

Finally we recall the following definition (see for example \cite[p.22]{BG}),
\begin{defn}
Let $\Omega$ be an open subset of $\RR^2$. We say that $\Omega$ has a regular boundary of class $C^k$ ($k\geq 1$)
if for every $p\in\partial \Omega$ there is a neighborhood $U_p$ of $p$ and a diffeomorphism $\phi_p$ of class $C^k$ from $U_p$ onto a neighborhood $V_p$ of $0$ in $\RR^2$ such that $\phi_p(p)=0$, \[\phi_p(U_p\cap \overline{\Omega})=V_p\cap \{(x,y)\in\RR: x\leq 0\}\] and the Jacobian determinant $J(\phi_p)$ is $>0$ in $U_p$.
\end{defn}
We have  (see for example \cite[Proposition p. 27]{BG}) that for $\Omega$ a relatively compact, open subset of $\CC$ with piecewise regular boundary (of class $C^k$, $k\geq 1$), there is only a finite number of connected components of $\partial\Omega$ and each of them is a Jordan curve (piecewise $C^k$).

Taking $\OO_n=\{z, d(z, K)<\frac{1}{n}\}$, a corresponding ${\phi}_n$, we have $K=\cap \OO_n=\cap  \text{supp}({\phi}_n)$ and $(\mu_{\text{supp}({\phi}_n)})_n$ converges weakly  to  $ \mu_K$. If for $U $, an open set containing ${\text{supp}({\phi}_n)}$, there is a sequence $(P_{n,m})_m$ of monic polynomials with integer coefficients whose roots are in $U$ and are such that the associated zeros counting measure $\mu_{P_{n,m}}$ converge weakly to the equilibrium measure $ \mu_ {\text{supp}({\phi}_n)}$, then by the diagonal process we have the same result for $U $, an open set containing $ K$.

In conclusion we can always suppose that $K$ is a compact with $C^{\infty}$ boundary, then the number of connected components of $\partial K$ is finite.

In the following the regularity of $K$ is understood. Note first that which is important in $K$ is the boundary of the outer component of $\CC\backslash \partial K$. Then $K$ can be a compact whose boundaries of the bounded connected components are what ever you want.

Denoting by $\Omega$ the outer component of $\CC\backslash \partial K$, we know that the equilibrium measure of $K$ is supported on $\partial \Omega$ (see for example, \cite{Ts}, \cite{Simon} Theorem A.10). We denote by $\partial \Omega=(\Gamma_1,...,\Gamma_r)=\Gamma$.

For the sake of completeness, we recall the following results.
\subsection{ Hyperelliptic Riemann surface associated to $\Gamma$.}

Let $\Omega$ be a plane domain. We have  seen that we may assume that each boundary component of the boundary of $\Omega$, denoted by $\partial \Omega=\Gamma=(\Gamma_1, \Gamma_2,...,\Gamma_r)$, $\Gamma_j$ is a smooth analytic curve. Alternatively, one may think of $\Omega$ as a plane bordered Riemann surface.
More precisely
\begin{defn}
For each $r=1,2,...$ we shall denote by $\U_r$ the class of plane domains whose boundary consists of $r$ disjoint Jordan curves $\Gamma_1, \Gamma_2,...,\Gamma_r$ which satisfy the following smoothness condition: with each $\Gamma_j$ there is associated a function $z_j(t)$ analytic and univalent in a neighborhood of $\Gamma_j$ which maps this neighborhood onto the circular ring $1-\delta <|z|<1+\delta$ and the curve $\Gamma_j$ onto the circle $|z|=1$.
\end{defn}

 To be complete, we recall the following results.
\begin{defn}
A closed Riemann surface of genus $g$ is hyperelliptic if it admits an analytic involution with precisely $2g+2$ fixed points. Such an analytic involution is called a sheet interchange. All of the Weierstrass points on a hyperelliptic surface are located at the fixed points of the sheet interchange.
\end{defn} 
We have the  theorem (\cite{ba} Theorem 3 p.13).
\begin{thm}\label{t3}
On a hyperelliptic Riemann surface
\begin{enumerate}
\item there is only one sheet interchange, and
\item any two meromorphic functions of order two are related by a fractional linear transformation.
\end{enumerate}
\end{thm}

We will also occasionally need the following theorem (\cite{ba} Theorem 7 p.19).
\begin{thm}
Let $\Omega$ be a domain in the class $\U_r$, $r\geq 3$. Then the following statements are equivalent:
\begin{enumerate}
\item the double of $\Omega$ is a hyperelliptic Riemann surface,
\item the domain $\Omega$ can be mapped one-to-one conformally onto the exterior of a system of slits taken from the real axis,
\item the domain $\Omega$ admits an anticonformal involution possessing precisely $2r$ fixed points on the boundary of $\Omega$.
\end{enumerate}
\end{thm}
At first, notably to put on the notations, we recall some results on the double of the exterior of a system of slits taken from the real axis.
Let $E=\bigcup_{j=1}^r E_j, \quad E_j=[e_{2j-1}, e_{2j}]\subset \RR$, $e_1<e_2<e_3<...<e_{2r-1}<e_{2r}$.

We let $q(z)$ be the polynomial $q(z)=\prod_{i=1}^{2r}(z-e_i)$. Such a polynomial will be called the structure polynomial of $\CC\backslash E$. Consider the subset of $\hat{\CC}\times \hat{\CC}$ given by
$$W=\{(z,w), w^2=q(z)\}\, .$$
where we add two points at infinity $\infty_+$ and $\infty_-$, characterized by the fact that $\dfrac{w}{z^{r-1}}=1$ at $\infty_+$ and $-1$ at $\infty_-$.\\
It's a two-sheeted branched covering space of the sphere $P_1(\CC)$, branched at the $2r$ points $e_j$. It's a topological covering space of $\CC\backslash E$.\\
The covering map $\pi_E:W \longrightarrow P_1(\CC)$, $\pi:(z,w)\longrightarrow z$ is a meromorphic function of order two on 
$W$ whose only multiple points, each of multiplicity two, are located at the points $(e_j, q(e_j))$, $j=1, ...,2r$. Then $W$ is a hyperelliptic surface whose sheet interchange $T: W\longrightarrow W$ is given by $T: (z,w)\longrightarrow (z,-w)$.

Finally we exhibit $W$ as the double of the domain $\Omega$ (\cite{ba}).

\begin{thm}
The double of $\CC\backslash E$, 
 is conformally equivalent to the Riemann surface $W$:  
$$w^2-q(z)=0\, .$$
\end{thm}
\begin{proof}
Let denote by  $D=\CC\backslash E$. First observe that $D$ admits an analytic square root of its structure polynomial $\sqrt{q(z)}$. We want to define $\sqrt{q}$ as an analytic function on $D$, the branch with 
$$\sqrt{q(x)}>0 \quad \text{if } x>e_{2r}\, .$$
This implies
\begin{eqnarray}
\sqrt{q(x)}<0&  & (e_{2r-2}, e_{2r-1})\cup (e_{2r-6}, e_{2r-5})\cup ...\\
\sqrt{q(x)}>0&  & (e_{2r-4}, e_{2r-3})\cup (e_{2r-8}, e_{2r-7})\cup ...
\end{eqnarray}
$$ (-1)^r\sqrt{q(x)}>0 \quad \text{if } x<e_{1}\, $$
and ($\sqrt{q(x+i0)}$ means $\lim_{\epsilon\downarrow 0}\sqrt{q(x+i\epsilon)}$)
\begin{eqnarray}
(-i)\sqrt{q(x+i0)}>0&  & (e_{2r-1}, e_{2r})\cup (e_{2r-5}, e_{2r-6})\cup ...\\
i \sqrt{q(x+i0)}>0&  & (e_{2r-3}, e_{2r-4})\cup (e_{2r-7}, e_{2r-6})\cup ...
\end{eqnarray}

So that it is possible to write $W$ as the disjoint union $W=W^{°}\cup W^+\cup W^-$ where 
\begin{eqnarray*}
W°&=&\{(z,w)\in W, z\in\partial D\}\\
W^+&=&\{(z,w)\in W,w=\sqrt{q(z)},  z\in D\}\\
W^-&=&\{(z,w)\in W, w=-\sqrt{q(z)}, z\in D\}.
\end{eqnarray*}
Additionally the meromorphic function $\pi$ on $W$ defined by $\pi:(z,w)\rightarrow z$ effects a univalent mapping of both $W^+$ and $W^-$ onto $D$, and if $(z,w)\in W°$, then $\pi(z,w)\in\partial D$ and so in particular is real.

To obtain the double of $D$ we let $f$ be an analytic and univalent map of $D$ onto a domain $D'\in \U_n$, and we let $W'$ be the double of $D'$.  If $U'$ is the canonical anticonformal involution on $W'$ obtained by the doubling process, we have then the conformal homeomorphism of $W$ onto $W'$ given by
$$(z,w)\rightarrow \left\{\begin{array}{cc} f(z), & (z,w)\in W^+\\
                                                             f(z)=U'(f(\bar{z}), & (z,w)\in W\\
                                                              U'(f(\bar{z}), & (z,w)\in W^-
\end{array}\right.$$ 
and the proof is complete.
\end{proof}
It is easy to check that the sheet interchange corresponds to $T=UV=VU$, where $U=f^{-1}\circ U' \circ f$ and $V: D\rightarrow D, z\mapsto \bar{z}$.

We will also denote $W$ by $\widehat{\CC\backslash E}$.


Let us denote by $E_j^*$ the reciprocal image of $E_j$ in $\widehat{\CC\backslash E}$, $\displaystyle E^*=\bigcup_{j=1}^r E_j^*$  and $D_j^*$ the reciprocal image of $[e_{2j}, e_{2j+1}]$ for $j=1,...,r-1$. Then 
 $D_1^*,...,D_{r-1}^*, E_1^*,...,E_{r-1}^*$ form an homology basis of $\widehat{\CC\backslash E}$.

The adherence of $W^+$ is $W^+\cup W$ and will denote $E^*_+=W$ its border. We will define in the same manner $E^*_-$.

\subsection{The notion of  ``calibrated"}\label{cali}

In this section we refer in particular to \cite{Ro1}, \cite{edin}, \cite{Se} and the references within, the appellation calibrated corresponds to ``Pell-Abel type" in \cite{Se}. 

First, we recall the following terminology. The meromorphic 1-forms on a compact Riemann surface are called abelian differentials. The abelian differentials which are holomorphic will be called of the first kind; while the meromorphic abelian differentials with zero residues will be called of the second kind. Finally, a general abelian differential (which may have residues) will be called of the third kind. 

\smallskip

We denote by $w_{\infty_+,\infty_-}$ the differential of the third kind having a simple pole at $\infty_+$ and $\infty_-$,  with residue $-1$ and $+1$ respectively, normalized
by   
$$ \int_{e_{2j}}^{e_{2j+1}} w_{\infty_+,\infty_-}=0,\, \,j = 1, \cdots, r-1.$$

We recall now the link between $w_{\infty_+,\infty_-}$ and the Green function of $\CC \backslash E$ with pole at $\infty_+$, $g(z)$ (see for example \cite{Wi}, \cite{FS1}, and the references within):
\begin{enumerate}
\item $g(z)$ is harmonic in $\CC \backslash E $,\\
\item $g(z)-\ln|z|$ is harmonic in a neighborhood of $\infty^+$\\
\item $\lim_{z\to\zeta}g(z)=0$ for all $z\in E$.
\end{enumerate}
Consequently, $g(z)=\Re G(z)$ where $G$ is a holomorphic function uniquely determined up to the addition of a purely imaginary constant. A function $\widetilde{g}$, such that $\widetilde{g}(z)= \Im G(z)$ is called a harmonic conjugate of $g$. 
So $G$ is the multiple-valued function obtained by adding to $g$ its conjugate:
$$G(z)=\int_{\gamma_z}2\frac{\partial g}{\partial w}(w)\, dw=g(z)+i\widetilde{g}(z)\, , \text{ where } \gamma_z \quad \text{ path from a fixed point } z_1 \text{ to } z.$$

$$g(z)=\Re G(z)=\dfrac{G(z)+\overline{G(z)}}{2}\, .$$
You can find an expression of $G$ in terms of theta function in \cite{FS1}.

The derivative of this multi-valued function is clearly single-valued. Hence, $G'(z)$ is a holomorphic function in $\CC \backslash E $.\\

The functions $g$ and $G$ extend to all of $\widehat{\CC \backslash E }$ by reflecting across $E$, since g vanishes on $E$ and $\Re(G(z))=0$ on $E$: $g(U(z))=-g(z)$ and 
 $G(U(z))=-\overline{G(z)}$ for $z\in \CC \backslash E $ and $U$ the canonical anti-conformal involution of $\widehat{\CC \backslash E }$. In particular $dG=w_{\infty_+,\infty_-}$. From the theory of 1-differential form on $W$, $w_{\infty_+,\infty_-}$ is of the form
$$ w_{\infty_+,\infty_-}=\frac{\lambda^{r-1}+\sum_{k=0}^{r-2}c_k \lambda^k} {\sqrt{q(\lambda)}}\,.$$

 It is an abelian differential with poles at $\infty_+$ and $\infty_-$ and zeros at the $\lambda_j$  (the zeros of $G'(z)$) and $\lambda_j^*$ ( the points of the copy of $\CC \backslash E $ corresponding to the $\lambda_j\in \CC \backslash E $).
It is clear from the behavior of $g(x)$ on the real axis that there is exactly on such $\lambda_j$ in the $(e_{2j},e_{2j+1})$.

$$w_{\infty_+,\infty_-}= \displaystyle \frac{\prod_{j=1}^{r-1}(\lambda-\lambda_j)}{\sqrt{q(\lambda)}}\, d\lambda=i\, dp\, ,$$
where $dp$ is called the quasimomentum in \cite{Krichever}.
\begin{defn}
A compact set $E=[e_1,e_2]\cup [e_3,e_4]\cup \cdots \cup [e_{2r-1},e_{2r}]$ of
the real line is $N$-calibrated for some $N\in \NN,\, N\geq r$,  if the complex Green's function $G(z)$ 
 of $\widehat{{\BC}\setminus E}$ with pole at infinity 
satisfies the conditions 
\begin{equation} \label{one}
\int_{e_{2k}}^{e_{2k+1}} \frac{R(t)}{\sqrt{q(t)}} dt= 0,~~ \text{ i.e. } 
\int_{D^*_k}\frac{R(t)}{\sqrt{q(t)}} dt= 0,~~ 
 k=1, 2,\cdots r-1
\end{equation}
and
\begin{equation}\label{two}
\int_{e_{2k-1}}^{e_{2k}} \frac{R(t)}{\sqrt{q(t)}}  dt=\pm \frac{n_k \pi i}{N},~~ \text{ i.e. } 
\int_{E^*_k}\frac{R(t)}{\sqrt{q(t)}} dt=\pm \frac{2 n_k \pi i}{N},
\quad 
k=1, 2,\cdots r,
\end{equation}
where 
$$R(t)= (t-\lambda_1)(t-\lambda_2)\cdots(t-\lambda_{r-1}) $$
and $n_1,\cdots n_r, N \in {\BN},~~n_1+\cdots +n_r= N~.$
\end{defn}
 \noindent We observe that if a compact $E$ is  calibrated  with 
respect to $n_1, n_2,\cdots n_r, N$ it is also  calibrated with respect to 
$k n_1, k n_2,\cdots ,k n_r, k N,~k\in{\BN}^*$.

We have the following results (see for example, \cite{Se}, \cite{edin})
\begin{prop}
There exists signs $\epsilon_j\in \{-1,+1\}$ such that
$$\sum_{j=1}^r \epsilon_j \int_{e_{2j-1}}^{e_{2j}} w_{\infty_+,\infty_-}=i\pi\, .$$
\end{prop}

\begin{prop}
Let $E = [e_1,e_2]\cup\cdots\cup[e_{2r-1},e_{2r}]$ be an $N$-calibrated compact set. Define
$$f(z) = \cosh\!\left(N\int_{e_{2r}}^z \frac{R(t)}{\sqrt{q(t)}}\,dt\right).$$
Then $f$ is a polynomial of degree $N$. Moreover, if $A_N$ is chosen so that $T_N(z) = A_N f(z)$ is monic, then $T_N$ is the Chebyshev polynomial of degree $N$ of $E$.
\end{prop}
For the commodity of the lecture we will write the proof of the preceding proposition ( see for example \cite{Ro2}, \cite{Se}, \cite{desc}, \cite{edin}).
We use a result of
Ostrovskii, Pakovitch and  Zaidenberg \cite{OPZ}; we first recall some 
definitions. Let $D=D(a,r)$ be a closed disc centered at 
$a\in{\BC}$ and of radius r. We say that a compact $K\subset D$ 
supports D if D is the (unique) disc of smallest radius which 
contains $K$.
The classical inequality of Jung asserts 
that each compact convex $K$ of diameter $\delta$ is contained in a 
closed disc $D(a,\rho)$, with $\frac{\delta}{2}\leq\rho\leq\delta$. 
The following theorem is the main result in \cite{OPZ}

\begin{thm}\label{ostro}
Let $\Delta_r=\Delta(0,r)\subset {\mathbb C}$ be the disc of 
radius r centered at the origin, $K\subset \Delta_r$ be a supporting compact 
of $\Delta_r$, and $p\in{\BC}[z]$ be a monic polynomial of degree $n$. 
Then $p$ is the unique $n$-th polynomial 
of least deviation on $K_p=p^{-1}(K)$.
\end{thm}

\begin{proof}
Let $G(z)=\int_{e_{2r}}^z  \frac{R(t)}{\sqrt{q(t)}}\, dt $ because of (\ref{two}) $G$ is a multi valued function on $\widehat{\CC \backslash E }$, but as $\int_{\D_j^*}\frac{R(t)}{\sqrt{q(t)}}\, dt$ is in $2i\pi \ZZ$, then, by $2i\pi$ multiplicity of $\cosh$, $f$ is single valued in  $\widehat{\CC \backslash E }$, being meromorphic $f$ is of the form $\frac{R_1+\sqrt{q}}{R_2}$, where $R_1$ and $R_2$ are polynomial functions. Now it is easy to verify that $G(T(z))=G(z)$, where $T$ is the hyperelliptic involution,   then $f(T(z))=f(z)$, consequently $f$ is a rational function. 
Since $G(z)=\int_{e_{2r}}^z w_{\infty_+,\infty_-}$ the only poles of $G$ are 
$\infty_+,\infty_-$ and so it is also the case for $f$.
%
%

Hence $f$ is an entire function with a pole of order $r$ at $\infty$, then a polynomial of degree $r$.

As $f^{-1}([-1,1])=E$, we have the result.
\end{proof}

\subsection{A generalization of Serre's method, the case of a compact of $\RR$}

We have to notice that a compact $E$ of $\RR$ of capacity $\C(E)\geq 1$ is an example of symmetric domain with respect to the real axis. 

The following theorem 
 will be a consequence of  \cite[Theorem 1.6.2 ] {Se}.
\begin{thm}
Let $E$ be a compact of capacity $\C(E)\geq 1$, then for all neighborhood $U$ of $E$ in $\RR$, there exists a sequence of monic polynomials of degree $>0$, with coefficients in $\ZZ$, whose all roots are in $U$, $(P_n)$ such that $\lim \mu_{P_n}=\mu_E$.
\end{thm}

Before proving it, for the convenience of the reader we recall some facts (see for example \cite{desc}, Chapter 5).
\begin{prop}
Let $E\subset \RR$ be compact. Let
$${\widetilde E}_n= \{x\in\RR,\; \dist(x,E)\leq \frac{1}{n}\}\, .$$
Then
\begin{enumerate}
\item $E\subset...\subset {\widetilde E}_{n+1}\subset {\widetilde E}_n \subset ...\subset \RR$ \, and\;
$\displaystyle \bigcap_n \widetilde{E}_n=E$,
\item Each ${\widetilde E}_n$ is a finite union of disjoint closed  intervals.
\end{enumerate}
\end{prop}

\begin{proof}
The point (1) is clear. To show (2) we use the fact that every open set in $\RR$ is a countable union of disjoint open intervals. Hence $\RR\backslash E$ is a disjoint union of maximal open intervals. Since $E$ is compact, two of these intervals are unbounded and the others $\{J_k\}_{k\in I}$, where $I$ is a countable set,  are contained in the convex hull of $E$. Thus, all but finitely many $J_k$ lie in a given ${\widetilde E}_n$, showing $\RR\backslash {\widetilde E}_n$ is finite. Thus, all but finitely many $J_k$ lie in a given ${\widetilde E}_n$, showing $\RR\backslash {\widetilde E}_n$ has finitely many open intervals. It is easy to see that each of the finite disjoint closed intervals in ${\widetilde E}_n$ must have positive measure.

\end{proof}

Furthermore we have  \cite[ Theorem 3.9] {edin}, \cite{Se} and \cite[Theorem 5.6.1, p. 306]{desc} 

\begin{thm}
Let $\displaystyle E=\bigcup_{j=1}^{l+1} E_j$ be an $l$-gap set with $E_j=[\alpha_j,\beta_j]$, $\alpha_j<\beta_j<\beta_{j+1}$. Then for all $m$ large, there exist $l$-gap sets 
$\displaystyle E^{(m)}= \bigcup_{j=1}^{l+1} E_j^{(m)}$ with
\begin{enumerate}
\item $E_j\subset E_j^{(m)}$,
\item Each $E_j^{(m)}$ has harmonic measure in $E^{(m)}$ equal to $\displaystyle k_j^{(m)}/m$ with $k_j^{(m)}\in\{1,2,...,\}$,
\item For some positive constants $C_1, C_2$,
\begin{eqnarray}
|E_j^{(m)}\backslash E_j|&\leq& C_1 m^{-1}\\
\C(E)\leq \C(E^{(m)})&\leq &\C(E)+C_2 m^{-1}
\end{eqnarray}
\end{enumerate}
\end{thm}
Then we have
\begin{thm}\label{4.2}
Let $E\subset \RR$ be compact. Then there exist $E_n$ so that $E\subset...\subset E_{n+1}\subset E_n\subset ...\subset \RR$ and
$\displaystyle \bigcap_n E_n=E$ holds, and
\[E_n\subset E_{n-1}^{int}\]
and each $E_n$ is the spectrum of some two-sided periodic Jacobi matrix. Moreover
\begin{enumerate}
\item $\mu_{E_n}\to \mu_E$,
\item $\C(_n)\to \C(E)$.
\end{enumerate}
\end{thm}

Now, we know that $E_n$ is a finite union of calibrated intervals and then by \cite{edin},  \cite{Se}, there exists a sequence of monic polynomials  with coefficients in $\ZZ$, $P_{k, E_n}$, whose all roots are in $E_n$, such that $\lim_k \mu_{P_{k, E_n}}=\mu_{E_n}$. By the diagonal process $\mu_E=\lim_j\mu_{P_{k_j, E_{n_j}}}$.

\begin{rem}
\begin{enumerate}
\item In particular, in the case of a compact of $\RR$, we can release the condition of countinuous boundary in the theorem 9.1 (Bilu, Rumely) of \cite{a}. We state this theorem in general terms  in the following section.
\item  Note that \cite{Se} Theorem 1.6.2  is stronger, in the sense that all the roots of the polynomials in question, are in the compact $E$.
\end{enumerate}
\end{rem}

\section{Some approximations}
In this section we give some known theorems on approximation of equilibrium measure by counting measure. The following article (\cite{BPRS}) goes further.

First of all, Rumely (\cite{ru}), generalizing a Bilu's theorem (\cite{Bi}) obtained the following equidistribution result.
\begin{thm}
Suppose a compact set $K\subset \CC$ with continue boundary, has capacity $\C(K)=1$ and is stable under complex conjugation. Let $(\alpha_n)_{n\geq 1}$ a sequence of algebraic integers ($ \alpha_n\not=\alpha_m$ if $n\not = m$), $\alpha_n$ of degree $d_{\alpha_n}$ such that for all open $U$ containing $K$, there exists $n_0$ such that for all $n\geq n_0$, $\alpha_n$ with all its conjugate, $O(\alpha_n)$,  are in $U$. Let $\Delta_n$ the measure
$$\Delta_n=\frac{1}{d_{\alpha_n}}\sum_{\beta\in O(\alpha_n)}\delta_\beta\, ,$$
then the measures $\Delta_n$ converge weakly to the equilibrium measure of $K$, $\mu_K$.

\end{thm}

%

This theorem answers Serre's question in the case of a compact of capacity 1, but in the proof the role of the polynomials $P_n(z)=\Pi_{\beta\in O(\alpha_n)}(z-\beta)$ is not explicit.

The second result we want to point out is the following theorem due to Pritsker \cite[Theorem 2.3]{prit}.
\begin{thm} 
Given any positive Borel measure $\mu$, $0\leq \mu(\CC)\leq 1$, that is symmetric about real line, there is a sequence of complete sets of conjugate algebraic integers such that their counting measures $\tau_n$ converge weakly to $\mu$.
\end{thm}
In this theorem we consider then the particular case of $\mu_K$, the equilibrium measure of a compact $K$ symmetric with respect to the real axis.

The proof needs in particular,  some  results on approximation of a finite set of points by algebraic integers.

In \cite{Mot}, Theorem 3.2, p.158 (see also \cite{du} for its effective version), it is proving that for every $n$ given numbers $z_1,...,z_{n}$, and every $\epsilon>0$, there exists an irreductible equation with complex integral coefficients $\alpha_1,...,\alpha_{n+1}$ and with roots $\zeta_{k,\epsilon}$ such that $|\zeta_{k,\epsilon} -z_k|<\epsilon$ for $k=1,...,n$. The same is true for real integral coefficients provided that the numbers $z_1,...,z_{n}$ are symmetric to the real axis.

For completeness we recall also the more general context of Ferguson's theorem.
In \cite{ferg} (Theorem A.1., p. 147) we have that if $\alpha_1$, ..., $\alpha_n$ are a complete set of conjugate algebraic integers over $\QQ[i]$, $\epsilon$ any positive number, and $z_2$, ...,$z_n$ any complex numbers. Then there is a polynomial $q\in \ZZ[i][z]$ such that
\[|q(\alpha_j)-z_j|<\epsilon, \quad 2\leq j\leq n.\]

A second ingredient of the proof of \cite[Theorem 2.3]{prit} is to express $\mu_K$ as the weak limit of counting measures . In fact,
there exist sequences of polynomials with their zeros in $K$, that verify
$\mu_n \to \mu_K$. We give an example which is linked with other interpretations of the capacity of a compact (see for example \cite{Ts}, \cite{Se} Appendix, \cite{Simon} Appendix B).

First of all, let $K\subset \CC$ be compact and infinite.  An $n$ point Fekete set is a set $\{z_j, j=1,...,n\}\subset K$ that maximizes
\[q_n(z_1,...,z_n)=\prod_{i\not=j}|z_i-z_j|\, .\]
 The Fekete constant (or diameter of $K$) is defined by
$$d_n(K)=q_n(z_1,...,z_n)^{1/n(n-1)}$$
for the maximizing set. $d_n(K)$ has a limit called the transfinite diameter of $K$, which is in fact, equal to $\C(K)$.

Secondly,  we recall that the Chebyshev polynomials, $T_n$, are defined as those monic polynomials of degree $n$ which minimize
$||z^n +p(z)||$, where $p\in \p_{n-1}$, the vectorial space of all polynomials of degree $\leq n-1$ and $||.||_K$ is defined by $||f||_K=\max_{z\in K}|f(z)|$. We have 
$\C(K)=\lim ||T_n||_K^{1/n}$.

From the maximum principle for analytic functions, it follows that for all $n$, the Fekete sets lie on the outer boundary of $K$. We have the following theorem.
\begin{thm}
The normalized density of Fekete sets converges to $\mu_K$, the equilibrium measure for $K$.
\end{thm}

\begin{rem}
Let $\{P_n\}$ be any sequence of monic polynomials having all their zeros in $K$ and such that the normalized zero counting measures for $P_n$ converge weakly to $\mu_K$. If $\partial K$ 
is regular (e.g., if it is connected), then 
\begin{enumerate}
\item $P_n$ are asymptotically optimal for the Chebyshev problem:
$$\lim_{n\to \infty}||P_n||_K^{1/n}=\C(K)\, .$$
If $\C(K)>0$ (so that $\mu_K$ is defined), then we also have:
\item Uniformly on compact subsets of the unbounded component of $\CC\backslash K$,
$$\lim_{n\to \infty}|P_n(z)|^{1/n}=\exp\{-\Phi_{\mu_K}(z)\}\, .$$
\end{enumerate}
\end{rem}


We come now to the proof that interest us.
We have then the following result (\cite{prit} p. 16). Let  a sequence of polynomials $(P_n)$ such that the corresponding normalized zero counting measure for $P_n$, $\mu_n$ converge to $\mu_K$. We can suppose $deg(P_n)=n$ and denote its zeros by $z_1,...,z_n$. Therefore, we can approximate this measure by a sequence of the counting measures $\tau_{n+1}$ for the complete set of conjugate algebraic integers $\zeta_k=\zeta_{k,1/n}, \quad k=1,...,n+1$ by using the theorem of Motzkin. For any $n\in \NN$, we approximate each point $z_k$ as close as we wish by one of the conjugate algebraic integers $\zeta_k$, $1\leq k\leq n$, obtained from Motzkin's theorem, while let the remaining $(n+1)$th conjugate algebraic integer $\zeta_{n+1}\to \infty$ as $n\to \infty$ (see \cite{du} p.160-161 for details). It follows that the resulting measures
$$\tau_{n+1}=\frac{1}{n+1}\sum_{k=1}^{n+1}\delta_{\zeta_k}$$
converge to $\mu_K$ as $n\to \infty$.

In fact $\displaystyle \mu_n=\frac{1}{n}\sum_{k=1}^n \delta_{z_k}\to \mu_K$, then $\displaystyle \nu_n=\frac{1}{n}\sum_{k=1}^n \delta_{\zeta_k}\to \mu_K$. For $f$ continous on $K$ and $n$ sufficiently large we obtain
\[\int f\tau_{n+1}=\int f\frac{n}{n+1}\nu_n=\frac{n}{n+1}\int f\nu_n \to \mu_K.\]

But as  already seen, the polynomial $Q_n$ with roots the complete set of conjugate algebraic integers $\zeta_k$, has $\zeta_{n+1}$ not in $K$ (for $n$ sufficiently large). So the sequence $(Q_n)$ doesn't answer Serre's suggestion (Question \ref{s}).


Finally we signal the link between orthogonal polynomials and equilibrium measure. 
Let K a compact of capacity $\C(K)>0$, and $d\rho$ be a measure with $\text{supp}(d\rho)=K$. Then we can form the uniquely existing orthonormal polynomials
$$p_n(z)=\gamma_n z^n+....,\quad \gamma_n>0, n\in\NN,$$
with respect to $\rho$. $\gamma_n$ is called the leading coefficient of $p_n$. The measure $\rho$ is said to be regular if $\lim \gamma_n^{1/n}=\dfrac{1}{\C(K)}$.
We have the following theorem (\cite{Simon} Thm 1.7, \cite{S-T} Thm 2.2.1).
\begin{thm}
Let K a compact of capacity $\C(K)>0$, such that $\Omega$ is dense in $\CC$. Let $d\rho$ be a measure with $\text{supp}(d\rho)=K$. If $\rho$ is regular, then $d\mu_{p_n}$ converges weakly to $d\mu_K$.

\end{thm}

\section{JACOBI}\label{jac}

In this paragraph, we recall some results on Jacobi operators, that we will need further.

A densely defined operator, $A$, on a Hilbert space $\mathcal{H}$, has a domain 
$D(A) \subset \mathcal{H}$, a dense subspace, and is a linear map of $D(A)$ into $\mathcal{H}$. 
Associated to $A$ is its graph, $\Gamma(A) \subset \mathcal{H} \times \mathcal{H}$, defined by
\(
\Gamma(A)=\{(\varphi,A\varphi)\mid \varphi\in D(A)\}.
\)
$\Gamma(A)$ is always a subspace of $\mathcal{H}\times \mathcal{H}$. 
$A$ is called closed if and only if $\Gamma(A)$ is closed.
$B$ is an extension of $A$ if and only if $\Gamma(A)\subset \Gamma(B)$, 
that is, $D(A)\subset D(B)$ and $B \upharpoonright D(A)=A$.

Given an operator $A$, we define $D(A^*)$ to be those $\varphi\in\mathcal{H}$ for which 
there is an $\eta\in\mathcal{H}$ with
\[
\langle \eta, \gamma\rangle = \langle \varphi, A\gamma\rangle
\qquad \text{for all }\gamma \in D(A). 
\]

$\eta$ is uniquely determined if it exists, since $D(A)$ is dense. 
We then set $\eta = A^* \varphi$, so
\[
\langle A^*\varphi,\gamma\rangle = \langle \varphi,A\gamma\rangle
\qquad \text{for all }\gamma\in D(A),\ \eta\in D(A^*).(*)
\]

$A^*$ is called the adjoint of $A$.
$A^*$ is thus defined to be the maximal operator so that $(*)$ holds. Moreover we have that an operator on a
Hilbert space is closable if and only if $A^*$ is densely defined, and in that case, $A^*$ is closed and its closure,
the smallest closed extension of $A$, $\overline{A}$, is $\overline{A} = (A^*)^*$.
If $A$ is a bounded operator then $\sigma(A^*) = \overline{\sigma(A)}$.
In effect, if $B$ is an invertible operator on a Hilbert space then $B^*$ is invertible and
\(
(\mu I - A)^* = \overline{\mu} I - A^*.
\)

Given complex numbers $a_n, b_n$, $n \in \ZZ$, with $b_n \ne 0$ for all $n$, we associate the infinite tridiagonal  two-sided complex Jacobi matrix
\[
\begin{pmatrix}
 \ddots &\ddots & \ddots &\ddots & \ddots & \ddots & \ddots\\
 \ddots &\ddots & a_{-1} &b_{-1}  & 0& \ddots & \ddots\\
  \ddots &0 & b_{-1} &a_0 & b_0 & 0   & \cdots \\
  \ddots &\ddots & 0 &b_0 & a_1 & b_1 & 0      & \cdots \\
  \ddots &\ddots & \ddots &0   & b_1 & a_2 & b_2    & \ddots \\
 \ddots &\ddots & \ddots &\ddots & \ddots & \ddots & \ddots
\end{pmatrix}.
\tag{1.1}
\]

In the symmetric case $b_n, a_n \in \mathbb{R}$ for all $n$ one recovers the classical Jacobi matrix.
Denoting by $\mathscr{C}_0 \subset \ell^2(\ZZ)$ the linear space of finite linear combinations of the standard basis 
$(e_n)_{n\in\ZZ}$, we may identify via the usual matrix product a complex Jacobi matrix  with an operator
acting on $\mathscr{C}_0$. Its closure  is called the corresponding second-order difference operator or
Jacobi operator.
$J$ is bounded  is equivalent to $\sum |a_n| + |b_n|\leq \infty$.

Let a Jacobi matrix defining a bounded self-adjoint operator on $\ell^2(\mathbb{Z})$, the two-sided Jacobi operator $J$:
\[
J e_n = b_{n-1}e_{n-1} + a_n  e_n + b_n e_{n+1}, \qquad n \in \mathbb{Z},
\]
where $(e_n)$ is a standard basis in $\ell^2(\mathbb{Z})$.

Split the axis $\mathbb{Z}$ into two semi-axes $\mathbb{Z} = \mathbb{Z}_+(m) \cup \mathbb{Z}_-(m)$,  
$\mathbb{Z}_+(m)=\{ n\in\mathbb{Z}: n\ge m\}$,  
$\mathbb{Z}_-(m)=\{ n\in\mathbb{Z}: n\le m-1\}$.  
Then
\[
\ell^2(\mathbb{Z}) = \ell^2_+(m) \oplus \ell^2_-(m), 
\qquad 
\ell^2_{\pm}(m) = \ell^2(\mathbb{Z}_{\pm}(m)).
\]
Define operators $J_{\pm}(m) = P_{\pm}(m) J P_{\pm}(m)$, where $P_{\pm}(m)$ are the orthogonal projectors onto $\ell^2_{\pm}(m)$.

In other words, we represent the operator $J$ in block form:
\[
J =
\begin{pmatrix}
J_-(m) & b_m e_m\langle \cdot , e_{m+1} \rangle \\
b_m e_{m+1}\langle \cdot , e_m \rangle & J_+(m)
\end{pmatrix}.
\]

%

We use the representation:
\[
J =
\begin{pmatrix}
J_- & 0 \\
0 & J_+
\end{pmatrix}
+ b_0\{e_0(\cdot,e_1) + e_1(\cdot,e_0)\}.
\]
To the operator $J$ we associate  its resolvent $R(z)=(J-zI)^{-1}=(J-z)^{-1}$ holomorphic on the resolvent set of $J$, $\rho(J)=\CC \backslash \sigma(J)$, where $\sigma(J)$ is the spectrum of the Jacobi operator. Its essential spectrum
$\sigma_{ess}(J)$ is the spectrum less the isolated eigenvalues of finite multiplicity.

%
%
%
%
%
%

We have (see for exemple \cite{Te}), in particular
\begin{equation*}
\sigma_{ess}(J)=\sigma_{ess}(J_-)\cup\sigma_{ess}(J_+)\, .
\end{equation*}


\section{Demonstration of the theorem following Serre.}
We recall the context. 
Let $K$ be a compact with $C^{\infty}$ boundary , symmetric with respect to the real axis, and $\Omega$ be the unbounded component of $\CC\backslash K$ (i.e. the component of the complement of $K$ which contains infinity). The boundary $\partial \Omega$ consists of $r$ mutually exterior curves, $\partial \Omega=(\Gamma_1,...,\Gamma_r)=\Gamma$. We know that $\mu_K$ has support in $\partial \Omega$.
$K$ and $\partial \Omega$ have same capacity, same equilibrium measure and same potential function (see for example \cite{Ts}, p.61).
 The boundary of the outer component, is symmetric with respect to the real axis. 
Let
\[
 \Gamma^+ = \Gamma \cap \{ z,\, \Im z \ge 0 \}
 \quad\text{and}\quad
 \Gamma^- = \Gamma \cap \{ z,\, \Im z \le 0 \} = \overline{\Gamma^+}.
\]

\subsection{The case of the spectrum of a periodic Jacobi matrix}
\subsubsection{A compact set well calibrated.}
Let $J_r$ be a $r$-periodic Jacobi matrix, more precisely a two-sided $r$-periodic Jacobi matrix,  it defines a difference operator
of period $r$ with  $a,b \in l^{\infty}(\BZ)$,  $b_n\not =0,\ $ acting on $l^2(\ZZ)$,
\begin{equation}\label{def}
a_{n+r}=a_{n} , b_{n+r}=b_{n},\  n\in\BZ.
\end{equation}
Note that the integer $r$ used has no reason to be the same of the preceding section.

To simplify the notations, we write the operator $J_r$ as
$$J_r=\left( \begin{array}{cccccccccc}
 .& .& .& .& .&...&.& .& .& . \\
 .& . & .& .& .& ...& .& .&.& .\\
 .& .& a_1& b_1& .& ...& .& .& .& .\\
 .& .& b_1 & a_2 & b_2 &...& .&.&.&.\\
 .& .& .& .& .& ...& . b_{r-1}&.&.&.\\
 .& .& .& .& .& ...& a_r& b_r&. &.\\
 .& .& .& .& .&...&.& .& .& . 
\end{array}\right)$$
with $a_{i+r}=a_i\ , \ b_{i+r}=b_i\ ,\ i\in\BZ$.
We denote by
$\displaystyle {\BB} =\prod_{n=1}^{r}b_{n}$. 

 For each positive integer $n,$ we introduce the matrix 
$E_n=(e_{ij})_{i,j\in {\BZ}}$, $e_{ij}=1$ if $\vert i-j \vert =n$ 
and  $e_{ij}=0$ otherwise. In other words, 
$E_n=D^n+D^{-n},
\ e_{k,k-n}=e_{k,k+n}=1,$ for $k \in \BZ.$

We know from for example \cite{Na}, \cite{MM},\cite{edin},\cite{BGHT} that there exist a polynomial $ P_r$ such that $\sigma(J_r)= P_r^{-1}([-2,2])$. Let's be more precise.

Following the presentation of \cite{MM}, the periodicity of the Jacobi 
matrix $J_r$ is expressible by the commutation relation $J_rD^r=D^rJ_r$
and the eigenvalues $z$ and $h$ 
associated to a common eigenvector $f$ are elements of the curve
\begin{eqnarray*}
{\mathcal{R}}_0 &=& \{ (z,h) \in \CC\times \CC^*,\ Lf=zf,\ D^r f=hf
,\ f\not= 0\}\\
&=&\{ (z,h) \in \CC\times \CC^*,\ \det(C_h-zI)=F(h,h^{-1},z)=0\}
\end{eqnarray*}
where
$$C_h-zI=\left( \begin{array}{cccccc}
   a_1-z& b_1& .& ...& .& b_0 h^{-1}\\ b_1 & a_2-z & b_2 &...& .&.\\
   .& .& .& ...& .& b_{r-1}\\
   b_r h& .& .& ...& b_{r-1}& a_r-z\\
 \end{array}\right)\ .$$
It is easy to see that
$$F(h,h^{-1},\, z)=(-1)^{r+1} \{\prod_{i=1}^r b_i(h+h^{-1})-P(z)\},$$
where $P(z)$ is a  polynomial of degree $r$ with leading coefficient 1. More precisely  if

$$\Delta (i, j)=\begin{vmatrix}
a_i-z & b_i       &\cdots  &\cdots &\cdots &\cdots \\                         
b_i   & a_{i+1}-z & b_{i+1}&\cdots &\cdots &\cdots \\            
\cdots&\cdots     &\cdots  &\cdots  &\cdots &\cdots \\
\cdots&\cdots     &\cdots  &\cdots  &\cdots &\cdots \\

\cdots&\cdots & \cdots &b_{j-2} &a_{j-1}-z&b_{j-1}\\
\cdots&\cdots &\cdots  & \cdots &b_{j-1}  &a_{j}-z\\
\end{vmatrix}$$
 some long but simple computations yield to the following

$$P(z)=(-1)^{r-1} [ \Delta(2,r)-b_r^2 \Delta (2,r-1)]\, .$$
 The polynomial 
$ P_r$ in the preceding statement  is
$$ P_r(X)= \frac{1}{\mathcal B}P(X).$$
\noindent We will refer to ${\mathcal B} $ as the modulus of 
the Jacobi matrix $J_r$ and to $P_r$ as the Na\u\i man polynomial, of degree $r$.\\
\noindent We have the relation $ P_r(J_r)= E_r$ .  From \cite{Na} (Theorem 3), \cite{BG},  the spectrum $\sigma(J_r)$ of $J_r$ is the inverse image of 
$[-2\, ,\ 2]$ under $ P_r$. It is a compact set in $\BC$, it 
consists of  $r$ algebraic arcs which 
may occasionally have common end points ( \cite{Na}) . 
See also \cite{Nai}, \cite{te}.
In the case of one sided, see \cite{be}. 
For the following we recall some geometric
properties of inverse polynomial images
(see \cite{peher}, \cite{schief}). The inverse image,
\(
P_r^{-1}([-2,2])
\),
consists of \(r\) Jordan arcs, denoted
by \(\ \K_1,  \K_2, \ldots,  \K_r\), where on each \( \K_j\),
\(j=1,\ldots,r\),
\(
P_r(z)
\)
is strictly monotone decreasing from \(+2\) to \(-2\).

If \(z_0 \in \mathbb{C}\) is a zero of
\(
P_r^2-4
\)
of multiplicity \(m\),
then exactly \(m\) Jordan arcs
\(
 \K_{i_1}, \ldots,  \K_{i_m}
\)
of
\(
P_r^{-1}([-2,2]),
\qquad
1 \le i_1 < i_2 < \cdots < i_m \le r,
\)
have \(z_0\) as common endpoint.
Two arcs \( \K_j,  \K_k\), \(j\neq k\),
cross each other at most once.

Then
$\sigma(J_r)=\K_{1}\cup\K_{2}\cup...\cup\K_{r}$ with for all $i=1,...,r$, ${ P_r}^{-1}([-2,2])=\K_i$, we will sometimes denote it by $\K$.

 According to a transfer theorem
of Fekete \cite{Fekete}, we have that
$${\mathcal C} (\sigma(J_r))= {|\mathcal B|}^{\frac{1}{r}}.$$

We have $||P||_{\infty, K}=2|{\mathcal B}|.$
The Theorem \ref{ostro} shows clearly that the polynomial $ P(X)$ is the  Chebyshev 
polynomial of $ \sigma(J_r)$, the diameter
$[-2{\mathcal B}\ ,\ 2{\mathcal B}]$ 
supports the disc centered at
the origin, with radius $2|{\mathcal B}|$. 
\smallskip

Moreover (see for example \cite{peh}, Corollary 2),
\begin{prop}
We have 
$$P_r(z)=2\cosh(rG(z))\, ,$$
where $G$ is the complex Green function of $\CC\backslash \K$.

\end{prop}

\subsubsection{How to obtain rational coefficients}

Approximating
 $J_r$ by rational coefficients $a_{i,q},b_{i,q}\in \QQ[i]$, we can construct $J_{r,q}$ as close of $J_r$ as we want and so, corresponding $P_{q}$ as close of $P$ as we want, with the coefficient of $P_{q}$ in $\QQ[i]$. We have also $\mathcal B_q$  in $\QQ[i]$ ,  with the coefficient of the corresponding Na\u\i man polynomial
$P_{r,q}$ in $\QQ[i]$.

As
$\sigma(J_{r,q})=
 P_{r,q}^{-1}([-2, 2])$ is the spectrum of the periodic Jacobi operator $J_{r,q}$ and as $\forall \epsilon>0$, there exists $q_0$ such that $\forall q\geq q_0; ||J_r-J_{r,q}||\leq \epsilon$, we conclude that in any neighborhood of $\sigma(J_r)$ there exists such $\sigma(J_{r,q})$.

So, without loss of generality we can suppose that $\sigma(J_r)$ is the spectrum of a periodic Jacobi matric with coefficients in $\QQ[i]$, that we do in the following. 
%
Let us denote by $A_{j}^{\pm} $ the points $ P_r^{-1}\{\pm 2\}$, the end points of $\K_j$, $P$ achieves its extrema at the end points of $\K_{j}$. As the degree of $ P_r$ is $r$, $ P_r$ has an unique zero in each  $\K_{j}$, denoted by $\xi_j$, $j=1,...,r$ and $\forall j=1,...,r$ it is strictly monotone decreasing from $\K_j$ to $[-2,2]$.

%


Let $T_n$ the Chebyshev polynomial of $[-2,2]$, then the polynomial $P_n$ defined by
$P_n=\mathcal B^{n}T_n( P_r)$ is the Chebyshev polynomial of degree $nr$ of $K$. In fact it verifies   clearly   \(K=  2\mathcal B^{n}T_n( P_r)^{-1}[-2 \mathcal B^{n},2\mathcal B^{n}]
\) and $[-2 \mathcal B^{n},2\mathcal B^{n}]$ is a supporting compact of the disc of radius $2|\mathcal B|^{n}$ centred at the origin.

Let $B_{k,n}$ such that $ P_r(B_{k,n})=2\cos \frac{k\pi}{n}$, then $T_n(B_{k,n})=(-2)^k$ and $|P_n|$ achieves its maximum at $A_{j}^{\pm} $ and the  $B_{k,n}$. The polynomial $P_n$  has $n$ zeros in each  $\K_{j}$. It enables us to prove the lemma.
%
%
%
%

\begin{lemma}
 $\mu_{P_n}$ converge to the equilibrium measure of $\sigma(J_r)$.
\end{lemma}
\begin{proof}
We will use the notations of section \ref{defs}.
The proof follows \cite[Theorem B1]{Simon}. 
By the Bernstein-Walsh lemma we have for $z\in\CC\backslash \K$,
\begin{equation}\label{*}
\frac{1}{nr}\log|P_n(z)|\leq \log\left( \frac{||P_n||_{\K}^{\frac{1}{nr}}}{\C(\K)}\right)-\Phi_{\mu_{\K}}(z)
\end{equation}
with the remark  that $\lim_n ||P_n||_{\K}^{\frac{1}{nr}}=\C(\K)$.\\
Now let $d\mu$ be a limit point of $(d\mu_n)_n$, where $d\mu_n$ is the normalized density of zeros of $P_n(z)$. Let $f$ be continuous with support contained in $\CC\backslash \K$, then $\int f d\mu_n=0$ because $P_n$ has its zeros in $\K$ and  if $d\mu_n \to d\mu$, then $\int f d\mu=0$ and $\text{supp}(d\mu)\subset \K$. By the Upper Envelope Theorem \cite[Theorem A7]{Simon}
$$\liminf_{n\to \infty}\Phi_{\mu_n}(z)= \Phi_{\mu}(z)$$
for all $z\in\CC\backslash \K$ except for $z$ in a polar set. Then letting $n\to\infty$ in (\ref{*}), we get
$$\Phi_{\mu}(z)\geq \Phi_{\mu_{\K}}$$
for all $z\in\CC\backslash \K$ except for $z$ in a polar set. By continuity of $\Phi_{\mu}$ and $\Phi_{\mu_{\K}}$ in $\CC\backslash \K$, we have $\Phi_{\mu}\geq \Phi_{\mu_{\K}}$ for all $z$ in $\CC\backslash \K$. By using  \cite[ Theorem A.21]{Simon}, we have $\mu=\mu_{\K}$. Thus, $\mu_{\K}$ is the only limit point of $(\mu_n)_n$ and so the limit is $\mu_{\K}$.
\end{proof}


With the same proof we deduce that for any $l\geq 1$, $\mu_{P_n^l}\to \mu_{\K}$.

In the following one can note by $M_n=||P_n||_{\sigma(J_r)}=2|\mathcal B|^n$.

From the preceding study we will deduce the following estimation, we'll need soon.
\begin{lemma}\label{arc}
For all $k=0,...,n-1$ between $B_{k,n}$ and $B_{k+1,n}$, there exists exactly one zero of $P_n$, denoted by $C_{k+1,n}$.
Moreover, if $C_{i,n}, C_{i+1,n}$ are two consecutive zeros of $P_n$ among the $n$ zeros in a fixed arc $K_j$, then
\[
\widehat{C_{i,n}C_{i+1,n}} = O\!\left(\frac{1}{n^{1/r}}\right).
\]
\end{lemma}

\begin{proof}
We have 
$ P_r(B_{k,n})=2\cos \frac{k\pi}{n}$, then \(
P(B_{k+1,n}) - P(B_{k,n}) = O\!\left(\frac{1}{n}\right).
\)
By Taylor's formula we have
\[
P_n(B_{k+1,n})
=
P_n(B_{k,n})
+
\sum_{i=1}^{r}
\frac{(B_{k+1,n}-B_{k,n})^i}{i!}
\,P_n^{(i)}(B_{k,n}).
\]
Then
\(
B_{k+1,n}-B_{k,n}
=
O\!\left(\frac1{n^{1/r}}\right)
\) from that we deduce 
$\widehat{B_{k+1,n}B_{k,n}}=O(1/n^{1/r})$ and the result.
\end{proof}
\subsubsection{How to obtain integer coefficients.}
Now, following  \cite[p 164]{fs} and its notations we consider, with $K=nr$,
\[P_n(z)= z^K+\frac{\gamma_1 z^{K-1}+\gamma_2 z^{K-2}+...+\gamma_K}{m},\quad\;\;m\in \NN^*,\,\gamma_j\in\ZZ[i].\]

Let us denote by $b=a^K, a\in\NN$ and by $ c= b!\,m^b$, $c>b\geq a$,  then $\displaystyle (P_n(z))^c=z^{Kc}+...$ and the coefficients of $\displaystyle z^{Kc-1}, z^{Kc-2},...,z^{Kc-b}$ are in $\ZZ[i]$. We can then determinate complex numbers, 
$\lambda_1^{(1)}, ...,\lambda_K^{(1)}, \lambda_1^{(2)}, ...,\lambda_K^{(2)}$,...,$\lambda_1^{(c-a)}, ...,\lambda_K^{(c-a)}$ such that the polynomial
\begin{eqnarray}
\Gamma_a(z)&=&(P_n(z))^c +\sum_{i=1}^{c-a}\sum_{j=1}^K\lambda_j^{(i)}z^{K-j}(P_n(z))^{c-a-i}\\
&=&(P_n(z))^c +\Delta_a(z)
\end{eqnarray}
has its coefficients in $\ZZ[i]$, also all $\lambda's$ are in the absolute value less than 2. We have
\begin{eqnarray*}
|\Gamma_a(z)-(P_n(z))^c |&\leq &2\sum_{i=1}^{c-a}\sum_{j=1}^K |z|^{K-j}|P_n(z)|^{c-a-i}\\
&\leq &2\sum_{j=1}^K |z|^{K-j}\frac{|P_n(z)|^{c-a}-1}{|P_n(z)|-1}
\end{eqnarray*}
$$\frac{|\Gamma_a(z)-(P_n(z))^c |}{|(P_n(z))^c |}\leq 2 \left(\sum_{j=1}^K |z|^{K-j}\right)\frac{|P_n(z)|^{-a}-|P_n(z)|^{-c}}{|P_n(z)|-1}\, .$$
%
%
%
%

On the lemniscate $|P_n(z)|=R_2$ with $R_2>1$ for example $R_2=M_n^{\frac{1}{n}}$, then with $M=\displaystyle\max_{|P_n(z)|=R_2}\sum_{j=1}^K |z|^{K-j}$,
\begin{eqnarray}\label{r}
|\Gamma_a(z)-(P_n(z))^c |&\leq &\frac{2M}{R_2^a(R_2-1)}|P_n(z)^{c}|\\
&\leq &\frac{1}{2} |P_n(z)^{c}| <|P_n(z)^{c}|
\end{eqnarray}
for $a$ sufficiently large.

\smallskip
We have the following estimation.
\begin{lemma}
For $n$ sufficiently large, the lemniscate $\{|P_n(z)|=M_n\}$ is in an $\epsilon$-neighborhood of $\sigma(J_r)$.
\end{lemma}
\begin{proof}
From proposition \ref{arc} we have $P_n(z)=2\mathcal B^n \cosh(nr G(z))$, with $g(z)=\Re G(z)$ the Green function of $\CC\backslash \K$. Then 
$$|P_n(z)|\leq M_n \Leftrightarrow |\cosh(nr G(z))|\leq 1 \Leftrightarrow |\sinh(nr g(z)|\leq 1\, .$$
Then $ \{|P_n(z)|=M_n\}$ is a subset of $\left\{|g(z)|\leq \dfrac{Argsh 1}{nr}\right\}$ , which is 
in an $\epsilon$-neighborhood of $\sigma(J_r)$ for $n$ sufficiently large.
\end{proof}

 Now we have the following picture:

\bigskip

\includegraphics[width=0.75\textwidth]{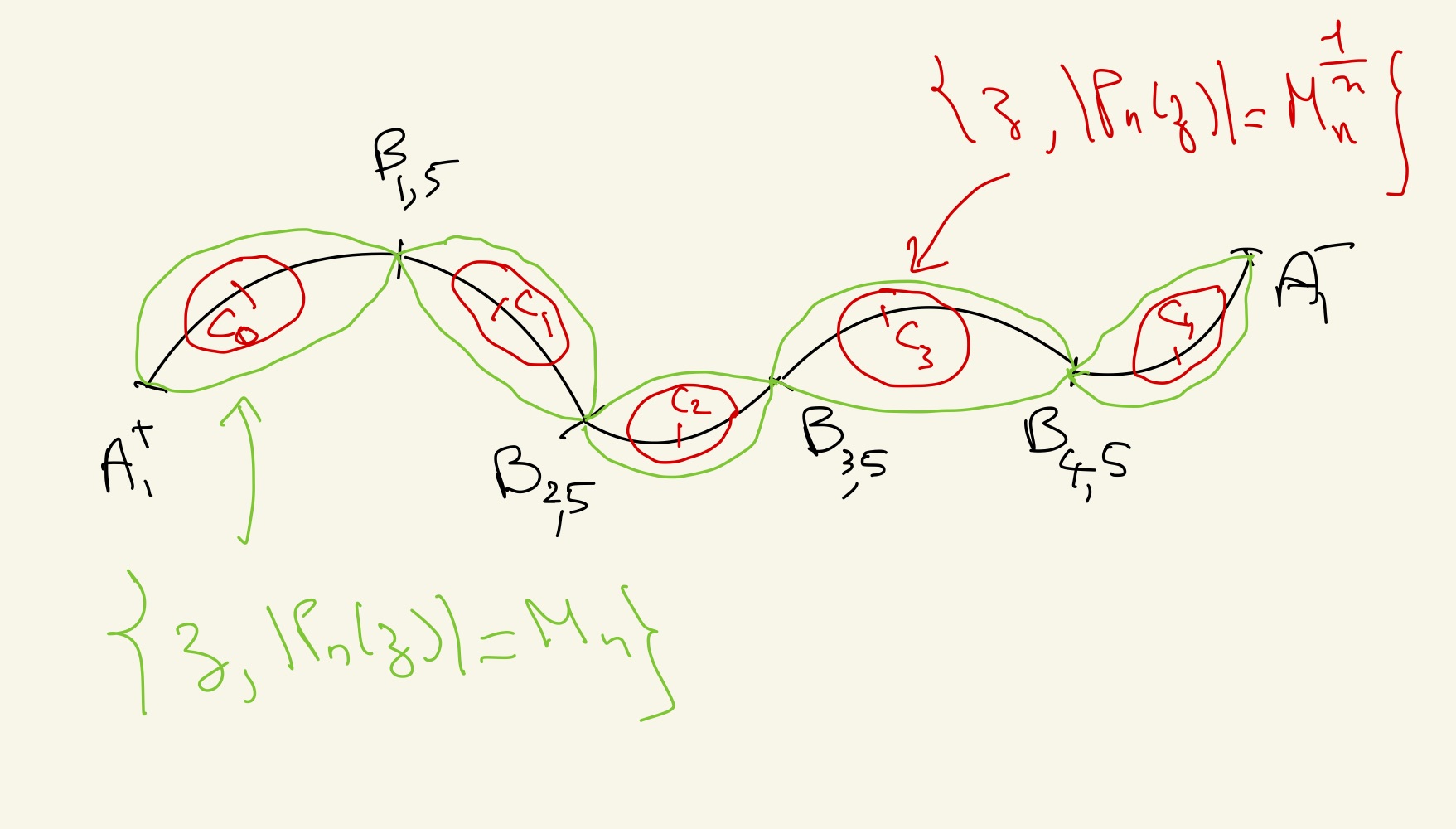}

\bigskip

\bigskip

For each $n$  let's denote by $Q_{nrc_n}$ a unitary polynomial with coefficient in $\ZZ[i]$, $\Gamma_a$, satisfying (\ref{r}).

By Rouch\'e's theorem, in each $\K_{j}$, $\Gamma_a$ has $nc_n$ zeros near from the $nc_n$ zeros of $(P_n(z))^{c_n}$.

We conclude that $\mu_{Q_{nrc_n}}-\mu_{P_n^{c_n}}=O(\frac{1}{n^{1/r}})$ and the result.


In conclusion, we have the following results.
\begin{thm}
If $\K = \sigma(J_r)$ is the spectrum of a two-sided periodic Jacobi matrix $J_r$
with $C(\K) \geq 1$, and if $U$ is an open set containing $\K$, there is a sequence $(P_n)$ of monic polynomials with coefficients in $\ZZ[i]$,  whose roots are in $U$ and are such that the associated zeros counting measure $\mu_{P_n}$ converge weakly to the equilibrium measure $ \mu_{\K }$ of $ \K$.

\end{thm}
For a one-sided periodic Jacobi matrix $J^{+}$, we associate the natural
two-sided periodic Jacobi matrix $J$.
The spectrum of a one-sided periodic Jacobi matrix has been studied,
for example in \cite{be}, Thm 2.7, and as
\(
\sigma_{\mathrm{ess}}(J) = \sigma_{\mathrm{ess}}(J^{+}),
\)
we have the following result.
\begin{cor}
Let  $\K^+$
the essential spectrum of a one-sided periodic Jacobi matrix $J_r$. If $U$ is an open set containing $\K^+$, there is a sequence $(P_n)$ of monic polynomials with  coefficients in $\ZZ[i]$,  whose roots are in $U$ and are such that the associated zeros counting measure $\mu_{P_n}$ converge weakly to the equilibrium measure $ \mu_{\K ^+}$ of $ \K^+$.

\end{cor}

\subsection{Approximation of $\delta \Omega$ by spectra of  periodic Jacobi matrix. The general case.}

In this section we will see that in every neighborhood of any finite union
of Jordan arcs $E$, we can find the spectrum of a two-sided periodic symmetric complex
Jacobi matrix whose capacity tends to that of $E$.
Doing this we will point out the following equivalence:\\
$\sigma$ is the spectrum of an $n$-periodic symmetric complex Jacobi matrix if
and only if there exists a polynomial of degree $n$ such that
\(
\sigma= T_n^{-1}([-1,1]).
\)
\begin{rem}
Notice that in the preceding assertion, ''the spectrum of an $n$-periodic symmetric complex Jacobi matrix" can be replace by
"the essential spectrum of a Jacobi operator associated with a symmetric tridiagonal matrix whose elements are asymptotically periodic". In effect,
let $\widetilde{A}$ be an operator with a periodic Jacobi matrix and let $A$ be a compact perturbation of $\widetilde{A}$ (for example, an asymptotically periodic Jacobi operator).
It is known from perturbation theory that any compact perturbation does not change the essential spectrum of a closed operator.
\end{rem}

We have yet seen the direct sense: see \S 7.2 with $P$ the Naiman
polynomial.
Now we will see the reciprocal sense in the following, where we also
demonstrate the announced result.
Notations are those of section \ref{7.3}.\\
Let $U$ be a neighborhood of $K$, then there exists $\varepsilon > 0$
such that $\Gamma^+(\varepsilon) \subset U$. We consider the Fekete
points of $\Gamma^+$, $\zeta_1^{(n)}, \dots, \zeta_n^{(n)}$, and the
associated Fekete polynomial
\[
T_n(z) = \prod_{k=1}^n (z - \zeta_k^{(n)}).
\]
Let
\[
\mu_n = \max_{\Gamma^+} |T_n(z)|.
\]
We have
\(
\forall n \quad \mu_n^{1/n} \geq C(\Gamma^+)
\quad \text{and} \quad
\lim \mu_n^{1/n} = C(\Gamma^+).
\)
Moreover, from \cite{fs} Thm G, the lemniscate domain defined by
\(
|T_n(z)| \leq \mu_n\, ,
\)
is contained in $\Gamma^+(\varepsilon)$ provided $n \geq n(\varepsilon)$.
Now we consider the polynomial
\(
{\mathcal T}_n(z) = \dfrac{1}{2\mu_n} T_n(z).
\)
From \cite{schiefer} Lemma 1 (and the references within), applied to our polynomials, we have the following
Pell-Abel equation.
Notice that not every Pell--Abel equation in $k[x]$, where $k$ is a field of characteristic different from 2, has a
solution, but every polynomial is a solution of a Pell--Abel equation
in $\mathbb{C}[x]$.
 (See also \cite{Se} p. 146--147 for a discussion on such equations.)
\begin{lemma}
For any polynomial ${\mathcal T}_n(z) = \frac{1}{2\mu_n} z^n + \cdots 
$, 
there exists a unique 
$\ell \in \{1,2,\dots,n\}$, a unique monic polynomial
\begin{equation*}
\mathcal{H}_{2\ell}(z) = \prod_{j=1}^{2\ell} (z - a_j) 
= z^{2\ell} + \cdots 
\end{equation*}
with pairwise distinct zeros $a_1, a_2, \dots, a_{2\ell}$, and a unique
polynomial $U_{n-\ell}(z) =  \frac{1}{2\mu_n}z^{\,n-\ell} + \cdots 
$
such that the polynomial equation
\begin{equation*}
{\mathcal T}_n^2(z) - 1 = \mathcal{H}_{2\ell}(z)\, U_{n-\ell}^2(z)
\end{equation*}
holds.
Further, there exists a monic polynomial 
$R_{\ell-1}(z) = z^{\ell-1} + \cdots \in \mathbb{P}_{\ell-1}$ such that
\begin{equation*}
{\mathcal T}_n'(z) = n\, R_{\ell-1}(z)\, U_{n-\ell}(z)
\end{equation*}
and, for $z \in \mathbb{C}$ with ${\mathcal T}_n(z) \notin [-1,1]$,
\begin{equation*}
{\mathcal T}_n(z) = \pm \cosh \left(
n \int_{a_j}^{z} \frac{R_{\ell-1}(w)}{\sqrt{\mathcal{H}_{2\ell}(w)}} \, dw
\right),
\end{equation*}
where $a_j$ is any zero of $\mathcal{H}_{2\ell}$.
\end{lemma}
{\it Note that the points $a_1, a_2,...,a_{2l}$ are exactly those zeros of ${\mathcal T}_n^2(z) - 1$ which have odd multiplicity.}
\begin{proof}
The first part is just an immediate consequence of the fundamental theorem of algebra (see \cite{schief}, Lemma 2). The second part is due to Peherstorfer (see \cite{peh}, Corollary 2).
\end{proof}
Now we see that $ {\mathcal T}_n^{-1}([-1,1])$ is the spectrum of a $n$ periodic Jacobi matrix.\\
We form  the hyperelliptic curve of equation
\(
\mu_n \bigl(h + h^{-1}\bigr) - T_n(z) = 0\, ,
\)
equivalently
\(
\omega^2 = 4\mu_n^2 \big( T_n^2(z) - 4\mu_n^2 \big).
\)
Remark that, if $1 \leq l\leq  n$, then $\mathcal{R}$ is singular.
\vspace{0.5cm}
From \cite{MM} it corresponds to a $n$-periodic complex Jacobi matrix $J$, such that
\(
T_n^{-1}\bigl([-2\mu_n, 2\mu_n]\bigr) = {\mathcal T}_n^{-1}([-1,1]) = \sigma(J).
\)
This spectrum of $J$ corresponds to at least $l$ Jordan arcs, moreover
\(
C(\sigma(J)) = \mu_n^{1/n}.
\)
For the commodity of the lecture 
%
%
%
%
we detail the example of \cite{MM} (Example 1 p.125) which corresponds to our case. Since $\mu_n \neq 0$,
\[
h(z) = \frac{1}{2\mu_n} \left( T_n(z) \pm \sqrt{T_n^2(z) - 4\mu_n^2} \right)
= \frac{1}{\frac{1}{2\mu_n} \left( T_n(z) - \sqrt{T_n^2(z) - 4\mu_n^2} \right)}.
\]
Therefore the curve is hyperelliptic of (arithmetic) genus
\(
g = n - 1,
\)
with two points $\infty_+$ and $\infty_-$ at infinity.\\
The hyperelliptic involution on the curve $\mathcal{R}$ maps
\(
(z,h) \mapsto (z,h^{-1}).
\) We denote it by $\tau$.
The meromorphic function $h$ has neither zeros nor poles except in the neighborhood of $z = \infty$.\\
When $z \to \infty$, we have
\(
h \sim \dfrac{\sqrt{T_n(z)}}{\mu_n}
\)
on sheet I, which shows that $h$ has a pole of order $N$.\\
Similarly, when $z \to \infty$,
\(
h \sim \dfrac{\mu_n}{T_n(z)}
\)
on sheet II.\\
Therefore the divisor $(h)$ of the function $h$ on the curve $\mathcal{R}$ is
\[
(h) = -n \infty_+ + n\infty_-.
\]
%
%
%
%
We return to the arguments as in \S 4.3.
Let $\{\alpha_j, \beta_j\}_{j=1}^g$ be a canonical homology basis of a
compact Riemann surface $\R$, satisfying
\[
\alpha_j \cdot \beta_k = \delta_{jk}, \qquad
\alpha_j \cdot \alpha_k = 0, \qquad
\beta_j \cdot \beta_k = 0.
\]
By Abel's theorem, the condition on the divisor
\(
(h) = -n \,\infty_+ + n \,\infty_-
\)
implies the existence of integers $n_j, m_j \in \mathbb{Z}$ such that
\[
\sum_{j=1}^g n_j \int_{\alpha_j} \omega
+ \sum_{j=1}^g m_j \int_{\beta_j} \omega
= n \int_{\infty_+}^{\infty_-} \omega
\]
for every holomorphic differential $\omega$ on $\R$.\\
Let $\eta$ be the normalized differential of the third kind with simple
poles at $\infty_+$ and $\infty_-$, with residues $-1$ and $+1$
respectively, normalized by
\(
\int_{\alpha_k} \eta = 0, \qquad k=1,\dots,g.
\)
Let $\{\zeta_1, \dots, \zeta_g\}$ be the basis of holomorphic
differentials dual to $\{\alpha_j\}$, i.e.
\(
\int_{\alpha_k} \zeta_j = \delta_{jk}.
\)
%
%
%
%
%
%
and 
\(
\eta = \frac{T_n' dz}{\sqrt{T_n^2 - 4\mu_n^2}},
\)
is the unique differential (up to scalar) with divisor $(\eta)=\D +\D^{\tau}-\infty_+-\infty_-$, where
\(
\mathcal D = \left(T_n'^{-1}(0)\right).
\) is the zero divisor on $\mathbb{c}$ of $T_n'$.
We want to apply this Theorem 2, of \cite{MM}
We will see that $\mathcal D$ is a regular divisor.
$\D$  is a positive divisor of degree  $g = n - 1$.
For a divisor $U$ on $\R$, we set $K(\R)$ the field of meromorphic functions
\(
L(U) = \{ f \in K(\R) : (f) + U \geq 0 \}
\)
and
\(
\Omega(-U) = \{ \text{meromorphic differentials } \omega \text{ on } \R : (\omega) \geq U \}.
\)
We will see that $\D$ is general, i.e.
\(
\dim L(\D) = 1
\)
and regular in the sense that for $k \in \mathbb{Z}$,
\[
\dim L(\D + k\infty_+ - (k+1)\infty_-) = 0.
\]
In the first step we show that $\D$ is general (see for example \cite{FK}, Prop. III.7.10).
Indeed, suppose there exists $f \neq 1$ in $L(D)$. By the Riemann–Roch theorem,
\[
\dim L(\D) = \deg \D - g + 1 + \dim \Omega(-\D).
\]
As $\deg \D = g$ and $\dim L(D) \geq 2$, we conclude
\(
\dim \Omega(-\D) \geq 1,
\)
and there is a holomorphic abelian differential $\omega$ such that $f\omega$ is also a holomorphic abelian differential.
Using the canonical differential basis, we conclude that $f$ is a rational function of $z$, hence a function of degree two, and so if $\nu\in\D$ then $(\nu,\nu^{\tau})\in \D\times\D$.
Contradiction with the definition of $\D$. Then $\dim L(D) = 1$.\\
In the second step we show that
\(
L(\D + k\infty_+ - (k+1)\infty_-) = \{0\}
\)
by induction.\\
For $k = 0$, $\dim L(\D) = 1$ and
\(
L(\D - \infty_-) \subset L(\D)
\)
since the function $f = 1$ belongs to the second space, but not the first.
\[
L\bigl(\D + k\infty_{+} - (k+1)\infty_{-}\bigr)
\subset
L\bigl(\D + k\infty_{+} - k\infty_{-}\bigr)
= \langle f_k \rangle
\]
Suppose
\(
\dim L\bigl(\D + k\infty_{+} - (k+1)\infty_{-}\bigr)=1.
\)
Then
\(
(f_k)+\D+k\infty_{+}-(k+1)\infty_{-}\ge 0.
\)
Define
\(
D^{(k)}=(f_k)+\D+k\infty_{+}-k\infty_{-}.
\)
Then
\(
D^{(k)}\ge 0
\qquad\text{and}\qquad
D^{(k)}-\infty_{-}\ge 0,
\)
that is,
\(
\infty_{-}\in D^{(k)}.
\)
As
\[
\deg D^{(k)}=\deg\ D=g,
\]
by the same preceding argument  we have also
\[
\infty_{+}\in D^{(k)}.
\]
Then
\(
D^{(k)} = U^{(k)} + \infty_{+},
\qquad
U^{(k)}\ge 0,
\)
and
\[
U^{(k)}
=
(f_k)+D+(k-1)\infty_{+}-k\infty_{-}
\ge 0,
\]
which contradicts the recurrence hypothesis.
Then for all $k$
\[
\dim L\bigl(\D+k\infty_{+}-(k+1)\infty_{-}\bigr)=0,
\]
and $\D$ is regula.
We can then apply Thm.~2 and $\R$ corresponds to an
$n$-periodic symmetric Jacobi operator $J$ with
\(
\sigma(J)=T_n^{-1}([-1,1]).
\)\\
$\sigma(J)$ is a subset of the lemniscate domain
\[
\{\, |T_n(z)| \leq 2\mu_n \,\}.
\]
With the same demonstration in Thm G of \cite{fs}, we have
\[
\sigma(J) \subset \{\, z : |T_n(z)| \leq 2\mu_n \,\} \subset \Gamma^{+}(\varepsilon)
\]
for $n \geq n_0$.
Then we can use lemma \ref{delta} to conclude that
\[
\exists K_n = \sigma(J_n), \text{ spectrum of a complex periodic Jacobi matrix, such that } \mu_{K_n} \to \mu_p^{+}.
\]

\subsection{Approximation of $\delta \Omega$ by spectra of  periodic Jacobi matrix. Remarks.\label{7.3}} 

\subsubsection{Some review of some spectra}\label{spec}

We recall some  definitions and properties, for more details see for example 
\cite{beclass}


An important part of the spectrum is the approximate point spectrum
$\sigma_a(A)$ that is defined as follows (see for example \cite{invit}).
\[
\sigma_a(A)
=
\left\{
\lambda \in \mathbb C :
\forall \varepsilon > 0\ \exists x \text{ with } \|x\| = 1
\text{ and } \|\lambda x - Ax\| < \varepsilon
\right\}.
\]
There are equivalent ways of
defining the approximate point spectrum.
\begin{align*}
\sigma_a(A)
&=
\left\{
\lambda \in \mathbb C :
\lambda - A \text{ is not bounded below}
\right\} \\
&=
\left\{
\lambda \in \mathbb C :
\exists \{x_n\} \subset X \text{ with } \|x_n\| = 1\ \forall n
\text{ and } \lambda x_n - Ax_n \to 0
\right\} \\
&=
\left\{
\lambda \in \mathbb C :
\exists \{x_n\} \subset X \text{ satisfying } x_n \neq 0
\text{ and } \lambda x_n - Ax_n \to 0
\right\} \\
&=
\left\{
\lambda \in \mathbb C :
\exists \{\lambda_n\} \subset \mathbb C \text{ with } \lambda_n \to \lambda
\text{ and } \exists \{x_n\} \subset X \text{ with }
\|x_n\| = 1\ \forall n
\text{ such that } \lambda_n x_n - Ax_n \to 0
\right\}.
\end{align*}

The points of the approximate point spectrum are also known as
approximate eigenvalues. Clearly,
\[
\sigma_p(A) \subset \sigma_a(A) \subset \sigma(A).
\]


\begin{proof}
Take an arbitrary $\lambda \in \sigma_c(A)$. By the definition of the continuous
spectrum, the range of $\lambda - A$ is dense and consequently (by Lemma~2.8)
the operator $\lambda - A$ is not bounded below. Therefore,
$\lambda \in \sigma_a(A)$.
\end{proof}

The important topological properties of the approximate spectrum are included
in the next theorem.

\begin{thm}
For an operator $A \in \mathcal L(X)$ we have the following:
\begin{enumerate}
\item The approximate point spectrum $\sigma_a(A)$ is a closed set.
\item Every boundary point of the spectrum belongs to the approximate point
spectrum. That is,
\[
\partial \sigma(A) \subset \sigma_a(A)
\quad
(\text{and so } \sigma_a(A) \neq \varnothing).
\]
\item The complement of $\sigma_a(A)$ in $\sigma(A)$, i.e., the set
\[
\sigma(A) \setminus \sigma_a(A)
=
\left\{
\lambda \in \mathbb C :
\lambda - A \text{ is bounded below and not surjective} \right\},
\]
is an open subset of the complex plane.
\end{enumerate}
\end{thm}

\subsubsection{An hypothesis}

%

Recall  $K$ is a compact such that $\delta \Omega=\Gamma$, the boundary of the outer component, is symmetric with respect to the real axis. 
Let
\[
 \Gamma^+ = \Gamma \cap \{ z,\, \Im z \ge 0 \}
 \quad\text{and}\quad
 \Gamma^- = \Gamma \cap \{ z,\, \Im z \le 0 \} = \overline{\Gamma^+}.
\]

The starting point of these remarks is the following result:


%

\begin{lemma}\label{bek}
There exist a two-sided Jacobi matrix, $J$, of the form

$$J=\left( \begin{array}{cccccc}
 .& .& .& .& .&. \\
 .& . &  \overline{b_1}& .& .& .\\
 .&  \overline{b_1}& \overline{a_1}& b_0& .& .\\
 .& .& b_0& a_1 & b_1 &.\\
 .& .& .& b_1&. & .\\
 .& .& .& .& .& .
\end{array}\right)$$

such that $\Gamma=\sigma_{\mathrm{ess}}(J)$.
More precisely,
\[
 \Gamma^+ = \sigma_{\mathrm{ess}}(J^+), \qquad \Gamma^- = \sigma_{\mathrm{ess}}(J^-)
\]
with the notations of  section \ref{jac}.
\end{lemma}

\medskip
\noindent\textbf{Proof.}
From \cite{beJ}, Example 5.2 and \cite{beC}, p.27,  there exists a Jacobi operator
\[
 J^+ = \begin{pmatrix}
  a_1 & b_1 \\
  b_1 & a_2 & b_2 \\
      & b_2 & a_3 & \ddots \\
      &     & \ddots & \ddots
 \end{pmatrix}
\]
such that $\sigma_{\mathrm{ess}}(J^+) = \Gamma^+$.


For the commodity of the lector we recall the demonstration.
[Example 5.2] \cite{beJ}
Let $E \subset \mathbb{C}$ be compact. Furthermore, let $(b_n)_{n \ge 0}$ be
dense in $E$, and suppose that for any isolated element $e$ of $E$ there
exists an infinite number of indices $n$ with $b_n = e$.
We consider the bounded linear operator $\tilde A$ with diagonal matrix
representation, i.e.,
\[
\tilde b_n = b_n \quad \text{and} \quad \tilde a_n = 0, \qquad n \ge 0.
\]
By construction, $b_k$ is an eigenvalue of $\tilde A$ for any $k \ge 0$,
with geometric multiplicity given by the multiplicity of $b_k$ in
$(b_n)_{n \ge 0}$. From \cite[Theorem IV.5.2]{kato}, we may conclude that the
range of $zI - \tilde A$ is not closed if for any $\varepsilon > 0$ there
exists an $n \ge 0$ with
\[
0 < |b_n - z| < \varepsilon.
\]
Also, from Lemma~5.1 \cite{beJ}, 
\(
\sigma_{\mathrm{ess}}(\tilde A) = E.
\)
Let $(a_n)_{n \ge 0}$ tend to zero. The operator $A$ resulting as $J^+$, is a
compact perturbation of $\tilde A$, and thus has the same essential
spectrum $\sigma_{\mathrm{ess}}(A) = E$ by
\cite[Chap.~IV, Theorem~5.35]{kato}.

\smallskip

Then
\[
 (J^+)^* = \begin{pmatrix}
  \overline{a_1} & \overline{b_1} \\
  \overline{b_1} & \overline{a_2} & \overline{b_2} \\
                  & \overline{b_2} & \overline{a_3} & \ddots \\
                  &                & \ddots & \ddots
 \end{pmatrix}
 = J^-
\]
and
\[
 \sigma(J^-) = \overline{\sigma(J^+)} = \Gamma^-.
\]

As by construction
\[
 \sigma(J) = \sigma_{\mathrm{ess}}(J)
           = \sigma_{\mathrm{ess}}(J^+) \cup \sigma_{\mathrm{ess}}(J^-)
           = \Gamma.
\]

We will show, under some hypothesis, that there exists a sequence of periodic Jacobi operators $(J_N^+)$ such that
$$\mu_{\sigma(J_N^+) }\rightarrow \mu_{\Gamma^+}\, .$$

Here is the

{\bf Hypothesis}:

 We suppose that there exists a Jacobi matrix
\[
 J^+ = \begin{pmatrix}
  a_1 & b_1 \\
  b_1 & a_2 & b_2 \\
      & b_2 & a_3 & \ddots \\
      &     & \ddots & \ddots
 \end{pmatrix}
\]
such that $\sigma_{\mathrm{ess}}(J^+) = \Gamma^+$ and 
$\limsup |b_1...b_n|^{1/n}=\C(\Gamma^+)$.

Notice that this is not the case for the Jacobi matrix constructed in the proof of  Lemma \ref{bek}.

We will need the following result:

\begin{prop}\label{P}

Let $A$ be a one-sided Jacobi operator such that $\sigma(A)=\partial\sigma(A)$, and let $(A_p)_p$ be the family of $p$--periodic Jacobi operators associated. 
$\forall \epsilon>0, \exists p_0, \forall p\geq p_0, \sigma(A_p)\subset \sigma_\epsilon(A)$, where $ \sigma_\epsilon(A)$ denotes the $\epsilon$ neighborhood of $\sigma(A)$.

\end{prop}


\begin{rem}
In the preceding Proposition \ref{P}, $\sigma(A)$ can be replace by $\sigma_{\mathrm{ess}}(A)$.
See for example \cite{beclass} Section 3.
\end{rem}

We are ready to prove Proposition \ref{P}, where we use the notations of section \ref{spec}.
\begin{proof}

%
%
%


Let $\varepsilon > 0$. We show  that $\exists p_0$ if $p \ge p_0$, then
\[
\sigma(A_p) \subset \sigma_\varepsilon(A)
\quad \Longleftrightarrow \quad
\mathbb{C} \setminus \sigma_\varepsilon(A) \subset \Omega(A_p),
\]
where $\Omega(A_p)$ is the resolvent set of $A_p$.

Assume that for all $ p_0$ $\exists p \ge p_0$, $\mathbb{C} \setminus \sigma_\varepsilon(A) \not\subset \Omega(A_p)$, i.e.
$z_p \in \sigma(A_p)\cap \mathbb{C} \setminus \sigma_\varepsilon(A) $. 

In particular, $z_p \in \sigma(A_p)$, thus by  Theorem 2.3 and Rem 2.4 (\cite{beC}, \cite{beclass}),

\[
\inf_n \frac{\|(z_p - A_p) y_n^{(p)}\|}{\|y_n^{(p)}\|} = 0,
\]
where 
$y_n^{(p)}=(q_0^{(p)}(z), q_1^{(p)}(z), ..., q_n^{(p)}(z), 0,0,...)$.

Since $\bigcup_p \sigma(A_p)$ is bounded, we may assume, without loss of generality,
that the sequence $(z_p)$ converges to some $z \in \overline{\bigcup_p \sigma(A_p)}$.

We also have $z_p \in \mathbb{C} \setminus \sigma_\varepsilon(A)$, hence
\[
\forall p,\quad d(z_p, \sigma(A)) \ge \varepsilon,
\]
and by continuity of the distance function,
\[
d(z, \sigma(A)) \ge \varepsilon.
\]

Note that for all $n < p+1$,
\[
y_n^{(p)} = y_n,
\]
and
\[
\lim_{p \to \infty} \frac{(z_p - A_p)y_n^{(p)}}{\|y_n^{(p)}\|}
= \frac{(z - A) y_n}{\|y_n\|}.
\]

We have on one hand
\[
\inf_p \left( \inf_n \frac{\|(z_p - A_p) y_n^{(p)}\|}{\|y_n^{(p)}\|} \right) = 0.
\]

On the other hand,
\[
\frac{(z_p - A_p) y_n^{(p)}}{\|y_n^{(p)}\|}
= \frac{(z_p - z) y_n^{(p)}}{\|y_n^{(p)}\|}
+ \frac{(z - A_p) y_n^{(p)}}{\|y_n^{(p)}\|},
\]
and $|z_p - z| \to 0$ as $p \to \infty$. Then
\(
\displaystyle\inf_p \frac{\|(z_p - A_p) y_n^{(p)}\|}{\|y_n^{(p)}\|}  =\frac{\|(z - A) y_n\|}{\|y_n\|}. 
\)

So we will have $ \displaystyle\inf_n\frac{\|(z - A) y_n\|}{\|y_n\|}=0$,
 we obtain the desired contradiction.
\end{proof}

Let $\Gamma^+$ verifying the Hypothesis. Then
the compact set $\Gamma^+$ and $\sigma(J_N^+)$ verify the hypothesis of Lemma \ref{delta}, we conclude that 
$$\mu_{\sigma(J_N^+) }\rightarrow \mu_{\Gamma^+}\, .$$

\subsubsection{An example}

A question is: is the Hypothesis always verified?

At least, we see, with an example \cite{boris}, that it can be verified.
 
\subsection{The end}

We will need the following result
\begin{lemma}
Let $K$ be a compact set of $\CC$, then
\begin{enumerate}
\item If $T_n$ is the Chebychev polynomial of $K$, then $\overline{T_n}$ is the Chebychev polynomial of $\bar{K}$,
 \item $\C(K)=\C(\bar{K})$.
\end{enumerate}
\end{lemma}
\begin{proof}
It is clear that the second point will follow from the first, but also from the fact that if $P_n$ a monic polynomial of degree $n$ then
 \[\Vert P_n||_{\infty,K}=||\overline{P_n}\Vert_{\infty,\bar{K}}.\]
\end{proof}
The second step is to verify that $\lim\mu_{\overline{P_n}}=\mu_{\Gamma^-}$ where $\Gamma^-=\Gamma\backslash \Gamma^+$.

The third step is to remark that $P_n\bar{P_n}$ is a monic polynomial with integer coefficients and that $\lim_n \mu_{P_n\bar{P_n}}=\frac{1}{2}(\mu_{\Gamma^+} +\mu_{\Gamma^-})$. Which follows from the definition of $\mu_{P_n\bar{P_n}}$ and the previous results.

The ultimate step is to verify that $\frac{1}{2}(\mu_{\Gamma^+} +\mu_{\Gamma^-})=\mu_\Gamma$. For this we look at the energy.
Let $\mu_1=\frac{1}{2}\mu_{\Gamma^+}$ and $\mu_2=\frac{1}{2}\mu_{\Gamma^-}$. We have $v(K)=I(\mu_K)=\ln(1/\C(K))$.

\begin{eqnarray*}
v(\Gamma^+\cup\Gamma^-)\leq I(\frac{1}{2}(\mu_{\Gamma^+} +\mu_{\Gamma^-}))&=&I(\mu_1)+I(\mu_2)+2\int\Phi_{\mu_1}(x)\, d\mu_2(x)\\
&=&\frac{1}{4}\ln\left(\frac{1}{\C(\Gamma^+)}\right)+\frac{1}{4}\ln\left(\frac{1}{\C(\Gamma^-)}\right) +2\int\Phi_{\mu_1}(x)\, d\mu_2(x)\, .
\end{eqnarray*}

From the beginning we have that $\C(\Gamma^+)=\C(\Gamma^-)=\C(\Gamma$). Moreover as $\Gamma^+$ and $\Gamma^-$ are disjoint we have that on $\Gamma^-$,
$\Phi_{\mu_1}(x)\leq \frac{1}{2}\ln(\frac{1}{\C(\Gamma^+)})$. Finally
$$I(\frac{1}{2}(\mu_{\Gamma^+} +\mu_{\Gamma^-}))\leq I(\mu_{\Gamma^+\cup\Gamma^-})=I(\mu_{\Gamma})\, ,$$
and by definition of the equilibrium measure $I(\frac{1}{2}(\mu_{\Gamma^+} +\mu_{\Gamma^-}))=I(\mu_{\Gamma})\, .$

This end the demonstration.

\end{document}